\documentclass[a4paper,12pt,reqno]{amsart}
\usepackage{amssymb,amsmath,amsfonts, amscd}
\usepackage{mathrsfs,latexsym}
\usepackage[all,ps,cmtip]{xy}
\usepackage[usenames,dvipsnames,svgnames,table]{xcolor}

\title{On projective threefolds of general type with small positive geometric genus}
\author{Meng Chen, Yong Hu, Matteo Penegini}

\address{\rm School of Mathematical Sciences, Fudan University, Shanghai 200433, China}
\email{mchen@fudan.edu.cn}

\address{\rm School of Mathematics, Korea Institute for Advanced Study, 85 Hoegiro, Dongdaemun-gu, Seoul 02455, Republic of Korea}
\email{yonghu11@kias.re.kr}

\address{\rm Universit\`a degli Studi di Genova, DIMA Dipartimento di Matematica, I-16146 Genova, Italy }
\email{penegini@dima.unige.it}

\thanks{The first author was supported by National Natural Science Foundation of China (\#12071078, \#11731004) and Program of Shanghai Subject Chief Scientist (\#16XD1400400). The second author is supported by a KIAS Individual Grant (MP062501)
at Korea Institute for Advanced Study. The third author was partially supported by PRIN 2015 ``Geometry of Algebraic Varieties'' and by GNSAGA of INdAM}

\newcommand{\bQ}{{\mathbb Q}}
\newcommand{\bP}{{\mathbb P}}
\newcommand{\roundup}[1]{\lceil{#1}\rceil}
\newcommand{\rounddown}[1]{\lfloor{#1}\rfloor}

\newcommand\lrw{\longrightarrow}

\newcommand\OO{{\mathcal{O}}}

\newcommand{\lsgeq}{\succcurlyeq}

\newcommand{\Mov}{\text{Mov}}

\newcommand{\simQ}{\sim_{\mathbb{Q}}}

\newtheorem{thm}{Theorem}[section]
\newtheorem{lem}[thm]{Lemma}
\newtheorem{cor}[thm]{Corollary}
\newtheorem{prop}[thm]{Proposition}

\theoremstyle{definition}
\newtheorem{defn}[thm]{Definition}

\newtheorem{exmp}[thm]{Example}

\newtheorem{rem}[thm]{Remark}
\theoremstyle{remark}

\begin{document}
\begin{abstract} 
In this paper we study the pluricanonical maps of minimal projective 3-folds of general type with geometric genus $1$, $2$ and $3$. We go in the direction pioneered by
Enriques and Bombieri, and other authors, pinning down, for low
projective genus, a finite list of exceptions to the birationality of
some pluricanonical map.  In particular, apart from a finite list of weighted baskets, we prove the birationality of $\varphi_{16}$, $\varphi_{6}$ and $\varphi_{5}$ respectively. 
\end{abstract}
\maketitle

\pagestyle{myheadings}
\markboth{\hfill M. Chen, Y. Hu and M. Penegini\hfill}{\hfill On projective 3-folds of general type with small positive geometric genus\hfill}
\numberwithin{equation}{section}

\section{\bf Introduction}

Studying geometric properties of pluricanonical divisors and pluricanonical maps of normal projective varieties is a fundamental aspect of birational geometry.  Indeed, the minimal model program (MMP for short; see, for instance, \cite{KMM, K-M, BCHM, Siu}) aims to find a model with nef canonical bundle and expects that the end result is either a Mori fiber space or a minimal model.  The remarkable theorem, proved separately by Hacon-McKernan \cite{H-M}, Takayama \cite{Ta} and Tsuji \cite{Tsuji}, says that there exists a constant $r_n$ (for any integer $n>0$) such that the pluricanonical map $\varphi_m$ is birational onto its image for all $m\geq r_n$ and for all minimal projective $n$-folds of general type.  The above mentioned number $r_n$ is an important quantity related to both boundedness problem and the explicit classification theory of varieties.  However, $r_n$ is non-explicitly given in general, except when $n\leq 3$ (namely, $r_1=3$, $r_2=5$ by Bombieri \cite{Bom} and $r_3\leq 57$ by Chen-Chen \cite{EXP1,EXP2, EXP3} and the first author \cite{Delta18}).

In this paper we investigate the birational geometry of projective 3-folds of general type with the geometric genus $p_g=1$, $2$ or $3$ by studying the birationality of their pluricanonical maps.

Let $V$ be a nonsingular projective 3-fold of general type. The 3-dimensional MMP suggests that one can replace $V$ by its minimal model $X$, provided that the property we are studying is birationally invariant.  By Chen-Chen's series of works in \cite{EXP1,EXP2,EXP3}, there exists a positive number $m_0\leq 18$ such that $P_{m_0}(X)=h^0(X, m_0K_X)\geq 2$. Hence it is possible to investigate the birational geometry of $X$ by studying the behavior of the $m_0$-canonical map $\varphi_{m_0, X}$.   This strategy has proved to be quite effective.

\begin{defn} Let $W$ be a $\bQ$-factorial normal projective variety of dimension $n$. Assume that $W$ is birational to a fibration  $g:W'\lrw S$, with  $W'$ being a nonsingular projective variety and $S$ being normal projective. Let us denote by $\tau$ the birational map $W\dashrightarrow W'$, by $n=\dim X$ and $s=\dim S$. 
Then we say that the set
$$\mathcal{F}=\{\hat{F}\subset W|\hat{F}=\tau_*^{-1}(F), F\ \text{is a fiber of}\ g\}$$
{\it forms an $(n-s)$-fold class} of $W$, where $\tau_*^{-1}(\cdot)$ denotes the strict transform. 
In particular, if $n-s=1$ ($=2$) we call it a \emph{curve class} (a \emph{surface class}). 
The number $\deg_{c}(\mathcal{F})=(K_W^{n-s}\cdot \tau_*^{-1}(F))$ ($F$ being general) is called {\it the canonical degree of $\mathcal{F}$}.  
\end{defn}


We shall use the above terminology in particular when $\varphi_{m_0,X}$ is of fiber type, i.e. $0<\dim \overline{\varphi_{m_0,X}(X)}<\dim X$. In this case $X'$ is a resolution of singularities of $X$. Moreover, we may assume that $X'$ is also a resolution of indeterminacy of $\varphi_{m_0, X}$, and $g$ is obtained by taking the Stein factorization: $ X' \stackrel{g}{\lrw} \Gamma \lrw \overline{\varphi_{m_0,X}(X)}$. In particular, we shall also say that $X$ is $m_0$-canonically fibred by a curve class  $\mathcal{C}$ (or a surface class $\mathcal{S}$).

Using the terminology just introduced, there are some relevant known results:
\begin{itemize}
\item[$\diamond$]  When $p_g(X)\geq 4$, $\varphi_{5,X}$ is birational; when $p_g(X)=3$, $\varphi_{6,X}$ is birational (see \cite[Theorem 1.2]{IJM}); when $p_g(X)=2$, $\varphi_{8,X}$ is birational (see \cite[Section 4]{IJM}).

\item[$\diamond$]  When $p_g(X)\geq 5$, $\varphi_{4,X}$ is non-birational if and only if $X$ is fibred by a genus two curve class of canonical degree 1 (see \cite[Theorem 1.3]{MZ}).

\item[$\diamond$] When $p_g(X)=4$, $\varphi_{4,X}$ is non-birational if and only if $X$ has possibly  $4$ birational structures described in \cite[Theorem 1.1]{MZ2016}.

\item[$\diamond$] When $p_g(X)=2$, $\varphi_{7,X}$ is non-birational if and only if $X$ is fibred by a genus $2$ curve class of canonical degree $\frac{2}{3}$ (see \cite[Theorem 1.1]{C14}).

\item[$\diamond$] When $p_g(X)=1$, $\varphi_{18,X}$ is birational (see \cite[Corollary 1.7]{EXP3}); when $p_g(X)=0$, $\varphi_{m,X}$ is birational for all $m\geq 57$ (see \cite[Theorem 1.6]{EXP3} and \cite[Corollary 1.2]{Delta18}).
\end{itemize}

On the other hand, the following examples give rise to, very naturally, some further questions.

\begin{exmp} \label{EXF}(\cite{F00}) Denote by $X_d$ a general weighted hypersurface in the sense of Fletcher. For instance,
\begin{itemize}

\item[(1)] $X_{12}\subset \bP(1,1,1,2,6)$ has the canonical volume $K^3=1$, the geometric genus $p_g=3$ and $\varphi_{5}$ is non-birational.

\item[(2)]  $X_{16}\subset \bP(1,1,2,3,8)$ has $K^3=\frac{1}{3}$, $p_g=2$ and $\varphi_7$ is non-birational;
 $X_{14}\subset \bP(1,1,2,2,7)$ has $K^3=\frac{1}{2}$, $p_g=2$ and $\varphi_6$ is non-birational.
 
 \item[(3)] $X_{28}\subset \bP(1,3,4,5,14)$ has $K^3=\frac{1}{30}$, 
 $p_g=1$ and $\varphi_{13}$ is non-birational.
\end{itemize}
\end{exmp}

\noindent{\bf Question A}. (see \cite[Problem 3.20]{MZ2016})  Let $X$ be a minimal projective 3-fold of general type.
\begin{itemize}
\item[(i)] When $p_g(X)=3$, is it possible to characterize the birationality of $\varphi_{5,X}$?
\item[(ii)] When $p_g(X)=2$, is it possible to characterize the birationality of $\varphi_{6,X}$?
\end{itemize}

The following conjecture is inspired by Example \ref{EXF}(1): 
\medskip

\noindent{\bf Conjecture B}.  Let $X$ be a minimal projective 3-fold of general type. When $p_g(X)=1$, then $\varphi_{14,X}$ is birational.
\medskip

The aim of this paper is to shed some light on the previous questions. In order to give a clear account for our main results, we need to recall the so-called ``weighted basket'' $\mathbb{B}(X)$, which is nothing but the triple $\{B_X, P_2(X),\chi(\OO_X)\}$ where $B_X$ is the Reid basket (cf. \cite{Reid87}) of terminal orbifolds of $X$.  

Before stating our main statements, let us fix the notation. By convention,  an ``$(l_1,l_2)$-surface''  means  a nonsingular projective surface of general type whose minimal model has the invariants: $c_1^2=l_1$ and $p_g=l_2$. Besides, we define ${\mathbb S}_1$ to be set of the following $5$ elements: 

{\tiny \begin{itemize}
\item[$\Diamond$] $B_1=\{4\times (1,2), (3,7), 3\times (2,5), (1,3)\}$,
 $K^3=\frac{2}{105}$;
 \item[$\Diamond$] $B_2=\{4\times (1,2), (5,12), 2\times (2,5), (1,3)\}$,
$K^3=\frac{1}{60}$;

\item[$\Diamond$] $B_3=\{7\times (1,2), (3,7), 2\times (1,3), (2,7)\}$,
$K^3=\frac{1}{42}$;

\item[$\Diamond$] $B_4=\{7\times (1,2), (3,7), (1,3), (3,10)\}$,
$K^3=\frac{2}{105}$;

\item[$\Diamond$]  $B_5=\{7\times (1,2), 2\times (2,5), 2\times (1,3), (1,4)\}$, $K^3=\frac{1}{60}$.
\end{itemize}}

Our first main result is the following:

\begin{thm}\label{main1} Let $X$ be a minimal projective 3-fold of general type with $p_g(X)=1$. Then
\begin{itemize}
\item[(i)]  $\varphi_{17,X}$ is birational.

\item[(ii)] $\varphi_{16,X}$ is birational unless $\chi(\OO_X)=P_2(X)=1$ and $B_X\in {\mathbb S}_1$.\end{itemize}
\end{thm}

In the second part, we mainly study the case with $p_g(X)=3$. Our second main result is the following:

\begin{thm}\label{thm_pg3} Let $X$ be a minimal projective 3-fold of general type with $p_g(X)= 3$. Then $\varphi_{5,X}$ is not birational onto its image if and only if either
\begin{itemize}
\item[(i)] $X$ is canonically fibered with genus $2$ curve fibres, and
$K_X^3 = 1$, or
\item[(ii)]  $X$ is canonically fibered with $(1,2)$-surface fibres
of canonical degree $\frac{2}{3}$ and ${\mathbb B}(X)$ 
belongs to an explicitly described finite set $ {\mathbb S}_3$.

\end{itemize}
\end{thm}

The idea of this paper naturally works for the case $p_g(X)=2$.  Being aware of the fact that the length of this paper would be too long to be tolerated by any journal. We would rather make the announcement here: 
\medskip

\noindent{\bf Theorem Z.} {\em Let $X$ be a minimal projective 3-fold of general type with $p_g(X)= 2$. Then $\varphi_{6,X}$ is not birational onto its image if and only if either
\begin{itemize}
\item[(i)] $X$ is canonically fibered with $(2,3)$-surfaces fibres of canonical degree $\frac{1}{2}$, or
\item[(ii)]  $X$ is canonically fibered with $(1,2)$-surface fibres and ${\mathbb B}(X)$ 
belongs to an explicitly described finite set $ {\mathbb S}_2$.
\end{itemize}}

\begin{rem} The existence of threefolds  described in Theorem \ref{thm_pg3}(i) and Theorem Z(ii) follows from Example 
\ref{EXF}. We do not know the existence of threefolds  described in Theorem \ref{main1}(ii), Theorem \ref{thm_pg3}(ii) and Theorem Z(i). A complete list of the elements of the sets ${\mathbb S}_3$  and ${\mathbb S}_2$  can be found at the following webpage.
\begin{center}
   \verb|http://www.dima.unige.it/~penegini/publ.html|
\end{center}   
\end{rem}

We briefly explain the structure of this paper.  In Section 2, we recall the established key theorem and some necessary inequalities. Section 3 contains some technical theorems which will be effectively  used for classification. Theorem \ref{main1} is proved in Section 4. Section 5 and Section 6 are devoted to proving Theorem \ref{thm_pg3}. 

In this paper we will be frequently and inevitably studying the canonical fibration $f:X'\lrw \bP^1$ of which the general fiber is a smooth $(1,2)$-surface. The two series of restriction maps $\theta_{m_1,-j}$ and $\psi_{m_1,-j}$ (see Definition \ref{twomaps}) give the decomposition of the pluri-genus, say
$P_m=\sum_{j\geq 0} u_{m,-j}$ for $2\leq m\leq 6$. The main observation of this paper is that, for each $j\geq 0$,
$\varphi_{6,X}$ (or $\varphi_{5,X}$) is birational when $u_{m,-j}$ is large enough. In other words, there are some constants $N_i>0$ ($2\leq i\leq 6$), $\varphi_{6,X}$ (resp. $\varphi_{5,X}$) is birational whenever $P_i\geq N_i$ for some $2\leq i\leq 6$.  Thus we are obliged to classify all those 3-folds of general type satisfying $P_i<N_i$ for all $2\leq i\leq 6$. Thanks to the orbifold Riemann-Roch built by Reid \cite{Reid87} and the basket theory established by Chen--Chen \cite[Section3]{EXP1}, we are able to do an effective classification.  

\section{\bf Preliminaries}

\subsection{Convention} For any linear system $|D|$ of positive dimension on a normal projective variety $Z$, we may write
$|D|=\text{Mov}|D|+\text{Fix}|D|$.   We say that {\it $|D|$ is not composed of a pencil if $\dim \overline{\Phi_{|D|}(Z)}\geq 2$}. A {\it generic irreducible element} of $|D|$ means a general member of $\text{Mov}|D|$ when $|D|$ is not composed of a pencil or, otherwise, an irreducible component in a general member of $\text{Mov}|D|$.

\subsection{Set up}\label{setup} Let $X$ be a minimal projective 3-fold of general type with $P_{m_0}(X)\geq 2$ for some integer $m_0>0$. Then the $m_0$-canonical map $\varphi_{m_0,X}: X\dashrightarrow \Sigma\subset \bP^{P_{m_0}-1}$ is a non-constant rational map, where $\Sigma=\overline{\varphi_{m_0, X}(X)}$. Fix an effective Weil divisor $K_{m_0}\sim m_0K_X$. Take successive blow-ups $\pi: X'\rightarrow X$ such that:
\begin{itemize}
\item[(i)] $X'$ is nonsingular and projective;

\item[(ii)] the moving part of $|m_0K_{X'}|$ is base point free;

\item[(iii)] the union of supports of both $\pi^*(K_{m_0})$ and exceptional divisors of $\pi$ is simple normal crossing.
\end{itemize}
Set $\tilde{g}=\varphi_{m_0}\circ\pi$ which is a morphism by assumption.
Let $X'\overset{f}\rightarrow \Gamma\overset{s}\rightarrow \Sigma$ be
the Stein factorization of $\tilde{g}$. 
We may write $K_{X'}=\pi^*(K_X)+E_{\pi}$ where $E_{\pi}$ is an effective ${\bQ}$-divisor which is supported on $\pi$-exceptional divisors. 
Set $|M|=\text{Mov}|m_0K_{X'}|$.
Since $X$ has at worst terminal singularities,
we may write $m_0\pi^*(K_X)\sim_{\mathbb Q}
M+E'$ where $E'$ is an effective ${\mathbb Q}$-divisor. Set
$d_{m_0}=\dim(\Gamma).$ Clearly one has $1\leq d_{m_0}\leq 3$.

If $d_{m_0}=2$, a general fiber of $f$ is a smooth
projective curve of genus $\geq 2$. We say that $X$ is {\it $m_0$-canonically fibred by curves}.

If $d_{m_0}=1$, a general fiber $F$ of $f$ is a smooth
projective surface of general type. We say that $X$ is {\it
$m_0$-canonically fibred by surfaces} with invariants $(c_1^2(F_0), p_g(F_0)),$ where $F_0$ is the minimal model of $F$ via the contraction morphism $\sigma: F\rightarrow F_0$. We may write $M\equiv a F$ where $a=\deg f_*\OO_{X'}(M)$. 

Let $S$ be a generic irreducible element of $|M|$.
For any positive integer $m$, $|M_m|$ denotes the moving part of $|mK_{X'}|$.
Let $S_m$ be a general member of $|M_m|$ whenever $m>1$.

Set
$$\zeta(m_0)=\zeta(m_0, |M|)=\begin{cases}
1, &\text{if}\ d_{m_0}\geq 2;\\
a, & \text{if}\ d_{m_0}=1.
\end{cases}$$
Naturally one has $m_0\pi^*(K_X)\simQ \zeta(m_0) S+E'.$ We define
\begin{align}\label{muS}
\mu=\mu(S)=\mathrm{sup}\{\mu'|\ \pi^*(K_X)\ge\mu'S \}.
\end{align}
Clearly we have $\mu(S)\geq \frac{\zeta(m_0)}{m_0}$.


\subsection{Known inequalities}\label{ss2} Pick a generic irreducible element $S$ of $|M|$. 
Assume that $|G|$ is base point free on $S$. Denote by $C$ a generic irreducible element of $|G|$.
 We define
$$
\beta=\beta(m_0, |G|)=\mathrm{sup}\{\beta'| \pi^*(K_X)|_S\ge\beta' C \}.
$$
Since $\pi^*(K_X)|_S$ is nef and big, we have $\beta>0$.

Define $$\xi=\xi(m_0, |G|)=(\pi^*(K_X)\cdot C)_{X'}. $$
For any integer $m>0$, we define
\begin{eqnarray*}
\alpha(m)=\alpha(m, m_0, |G|)&=& (m-1-\frac{1}{\mu}-\frac{1}{\beta})\xi,\\
\alpha_0(m)&=&\roundup{\alpha(m)}.
\end{eqnarray*}
We will simply use the simple notation $\zeta$, $\mu$, $\beta$, $\xi$ and $\alpha(m)$ when no confusion arises in the context.
According to \cite[Theorem 2.11]{EXP2}, whenever $\alpha(m)>1$, one has
\begin{equation} m\xi\geq \deg(K_C)+\alpha_0(m).\label{kieq1}
\end{equation}
In particular, Inequality \eqref{kieq1} implies
\begin{equation}
\xi\geq \frac{\deg(K_C)}{1+\frac{1}{\mu}+\frac{1}{\beta}}. \label{kieq2}
\end{equation}

Moreover, by \cite[Inequality (2.1)]{MA}, one has
\begin{equation}\label{kcube}
K^3_X \geq \frac{\mu \beta\xi}{m_0}.
\end{equation}

\subsection{Birationality principle} We refer to \cite[2.7]{EXP2} for birationality principle.  
We will tacitly and frequently use the following theorem in the context:

\begin{thm}\label{key-birat} (see \cite[Theorem 2.11]{EXP2}) Keep the same setting and assumption as in Subsection \ref{setup} and Subsection \ref{ss2}.  Pick up a generic irreducible element $S$ of $|M|$. For $m>0$, assume that the following conditions are satisfied:
\begin{itemize}
\item[(i)] $|mK_{X'}|$ distinguishes different generic irreducible elements of $|M|$;
\item[(ii)] $|mK_{X'}||_S$ distinguishes different generic irreducible elements of $|G|$;
\item[(iii)] $\alpha(m)>2$.
\end{itemize}
Then $\varphi_{m,X}$ is birational onto its image.
\end{thm}

\subsection{Variant}\label{var} Clearly, if we replace $|m_0K_X|$ with any of its non-trivial sub-linear system $\Lambda$  while taking $|M|$ to be the moving part of $\pi^*(\Lambda)$ and keeping the same other notations as in \ref{setup} and \ref{ss2}, Inequalities (\ref{kieq1}), (\ref{kieq2}) and Theorem \ref{key-birat} still hold.

\subsection{A weak form of extension theorem} Sometimes we use the following theorem which is a special form of Kawamata's extension theorem (see \cite[Theorem A]{KawaE}):

\begin{thm}\label{KaE}  (see \cite[Theorem 2.4]{MZ2016}) Let $Z$ be a nonsingular projective variety on which $D$ is a smooth divisor such that $K_Z+D\simQ A+B$ for an ample $\bQ$-divisor $A$ and an effective $\bQ$-divisor $B$ and that $D$ is not contained in the support of $B$. Then the natural homomorphism
$$H^0(Z, m(K_Z+D))\longrightarrow H^0(D, mK_D)$$
is surjective for all $m>1$.
\end{thm}

Take  $Z=X'$, $D=S$ and, without losing of generality, assume $\mu$ to be rational. We get 
$$|n(\mu+1)K_{X'}||_S\lsgeq |n\mu(K_{X'}+S)||_S=|n\mu K_S|$$
for some sufficiently large and divisible integer $n$. Noting that
$$n(\mu+1)\pi^*(K_X)\geq M_{n(\mu+1)}$$
and that $|n(\mu+1)\sigma^*(K_{S_0})|$ is base point free, we have
\begin{equation}
\pi^*(K_X)|_S\geq \frac{\mu}{\mu+1}\sigma^*(K_{S_0})\geq \frac{\zeta(m_0)}{m_0+\zeta(m_0)}\sigma^*(K_{S_0}).\label{cri}
\end{equation}

\subsection{Three lemmas on surfaces}  We need the following lemma in our proof.

\begin{lem}\label{ll1} (\cite[Lemma 2.6]{C14}) Let $S$ be a nonsingular projective surface. Let $L$ be a nef and big $\bQ$-divisor on $S$ satisfying the following conditions:
\begin{itemize}
\item[(1)] $L^2>8$;
 \item[(2)] $(L\cdot C_{x})\geq 4$ for all
irreducible curves $C_{x}$ passing through any very general point
$x \in S$.
\end{itemize}
Then  $|K_S+\roundup{L}|$ gives a birational map.
\end{lem}

\begin{lem}\label{ll2}(\cite[Lemma 2.4]{EXP3}) Let $\sigma:S\longrightarrow S_0$ be a birational contraction from a nonsingular projective surface $S$ of general type onto its minimal model $S_0$. Assume that $S$ is not a $(1,2)$-surface and that $C$ is a moving  curve on $S$. Then  $(\sigma^*(K_{S_0})\cdot C)\geq 2$.
\end{lem}

\begin{lem} \label{ll3}(\cite[Lemma 2.5]{EXP3}) Let $\sigma:S\longrightarrow S_0$ be the birational contraction onto the minimal model $S_0$ from a nonsingular projective surface $S$ of general type. Assume that $S$ is not a $(1,2)$-surface and that $C$ is a curve passing through very general points of $S$. Then one has $(\sigma^*(K_{S_0})\cdot C)\geq 2.$
\end{lem}

\subsection{The weighted basket of $X$}\label{basket}

The {\it weighted basket}  $\mathbb{B}(X)$ is defined to be the triple
$\{B_X, P_2(X), \chi(\OO_X)\}$.
We keep all the definitions and symbols in \cite[Sections 2 and 3]{EXP1} such as ``basket'', ``prime packing'',  ``the canonical sequence of a basket'', $\Delta^j(B)$ ($j>0$),  $\sigma$, $\sigma'$, $B^{(n)}$ ($n\geq 0$), $\chi_m(\mathbb{B}(X))$ ($m\geq 2$), $K^3(\mathbb{B}(X))$,  $\sigma_5$, $\varepsilon$, $\varepsilon_n$ ($n\geq 5$) and so on.

As $X$ is of general type, the vanishing theorem and Reid's Riemann-Roch formula \cite{Reid87} (see also front lines in \cite[4.5]{EXP1}) imply that
$$\chi_m(\mathbb{B}(X))=P_m(X)$$
for all $m\geq 2$ and $K^3(\mathbb{B}(X))=K_X^3$.  For any $n\geq 0$, $B^{(n)}$ can be expressed in terms of $\chi(\OO_X)$, $P_2$, $P_3$, $\cdots$, $P_{n+1}$ (see \cite[(3.3)$\sim$(3.14)]{EXP1} for more details), which serves as a considerably powerful tool for our classification.

\section{\bf Some technical theorems}

\subsection{Two lemmas on distinguishing properties}
\begin{lem}\label{s1} Let $X$ be a minimal projective 3-fold of general type with $p_g(X)>0$ and $P_{m_0}\geq 2$. Keep the setting in \ref{setup}. Then the linear system $|mK_{X'}|$ distinguishes different generic irreducible elements of $|M_{m_0}|$ whenever $m\geq m_0+2$.
\end{lem}
\begin{proof}  Since $mK_{X'}\geq M_{m_0}$, by the Matsuki-Tankeev birationality principle (see, e.g. \cite[2.1]{MPCPS}), 
it is sufficient to treat the case when $|M_{m_0}|$ is composed of a pencil.  Indeed, if $\Gamma\cong \bP^1$, global sections of $f_*\OO_{X'}(M_{m_0})$ (as a line bundle) distinguishes different points of $\Gamma$. Hence $|M_{m_0}|$ distinguishes different smooth fibers of $f$, so does  $|mK_{X'}|$.  {}From now on, we assume that $|M_{m_0}|$ is composed 
of an irrational pencil. Pick two generic irreducible elements $S_1$ and $S_2$.  The vanishing theorem (\cite{KaV,VV}) gives the surjective map:
\begin{eqnarray}
&&H^0(X', K_{X'}+\roundup{(m-m_0-1)\pi^*(K_X)}+M_{m_0}) \notag\\
&\lrw& H^0(S_1, \big(K_{X'}+\roundup{(m-m_0-1)\pi^*(K_X)}+M_{m_0}\big)|_{S_1})\label{ee1}\\
&&\oplus  H^0(S_2, \big(K_{X'}+\roundup{(m-m_0-1)\pi^*(K_X)}+M_{m_0}\big)|_{S_2})\label{ee2}
\end{eqnarray}
Both groups in (\ref{ee1}) and (\ref{ee2}) are non-zero as $S_i$ is moving and $M_{m_0}|_{S_i}\sim 0$. \end{proof}

\begin{lem}\label{s2}  Let $X$ be a minimal projective 3-fold of general type with $p_g(X)>0$ and $P_{m_0}(X)\geq 2$.  Keep the setting and notation in \ref{setup} and \ref{ss2}.  Then $|mK_{X'}||_S$ distinguishes different generic irreducible elements of $|G|$ under one of the following conditions:
\begin{itemize}
\item[(1)] $m>\frac{1}{\mu}+\frac{2}{\beta}+1$.
\item[(2)] $m>\frac{m_0}{\zeta}+m_1+1$ where the positive integer $m_1$ satisfies $M_{m_1}|_S\geq G$.
\end{itemize}\end{lem}
\begin{proof}  Without loss of generality, we may and do assume that $\mu$ is rational.

(1).  As $(m-1)\pi^*(K_X)-S-\frac{1}{\mu}E_S'\equiv (m-\frac{1}{\mu}-1)\pi^*(K_X)$ is nef and big and it has {\it snc} support by assumption,  the vanishing theorem gives
\begin{eqnarray}
|mK_{X'}||_S&\lsgeq& |K_{X'}+\roundup{(m-1)\pi^*(K_X)-\frac{1}{\mu}E_S'}||_S\notag\\
&\lsgeq& |K_S+\roundup{\big((m-1)\pi^*(K_X)-S-\frac{1}{\mu}E_S'\big)|_S}|. \label{dd1}
\end{eqnarray}
By assumption, we write $\frac{1}{\beta}\pi^*(K_X)|_S\equiv C+H$ for an effective $\bQ$-divisor $H$ on $S$.  Pick another generic irreducible element $C'$ of $|G|$.
 Similarly since
$$\big((m-1)\pi^*(K_X)-S-\frac{1}{\mu}E_S'\big)|_S-C-C'-2H\equiv (m-1-\frac{1}{\mu}-\frac{2}{\beta})\pi^*(K_X)|_S$$
is nef and big,  the vanishing theorem implies the surjective map:
\begin{eqnarray}
&&H^0(S, K_S+\roundup{\big((m-1)\pi^*(K_X)-S-\frac{1}{\mu}E_S'\big)|_S-2H}\notag\\
&\lrw& H^0(C, K_C+D_m)\oplus H^0(C', K_{C'}+D'_m)\label{dd2}
\end{eqnarray}
where $D_m=(\roundup{\big((m-1)\pi^*(K_X)-S-\frac{1}{\mu}E_S'\big)|_S-C-C'-2H}+C')|_C$
satisfying $\deg(D_m)>0$ and, similarly, $\deg(D'_m)>0$.  Thus both groups $H^0(C, K_C+D_m)$ and  $H^0(C', K_{C'}+D'_m)$ are non-zero. Relations (\ref{dd1}) and (\ref{dd2}) imply that $|mK_{X'}||_S$ distinguishes different generic irreducible elements of $|G|$.

(2).  By virtue of Relation (\ref{dd1}) (while replacing $\frac{1}{\mu}E_S'$ with $\frac{1}{\zeta}E'$), we may consider the linear system
$$|K_S+\roundup{\big((m-1)\pi^*(K_X)-S-\frac{1}{\zeta}E'\big)|_S}|.$$

Note that $p_g(X)>0$ implies $p_g(S)>0$.

When $\zeta=1$, we clearly have
\begin{eqnarray}
&&|K_S+\roundup{\big((m-1)\pi^*(K_X)-S-E'\big)|_S}|\notag\\
&\lsgeq& |K_S+\roundup{\big((m-m_1-1)\pi^*(K_X)-S-E'\big)|_S}+G|\label{yx1}
\end{eqnarray}
and
$\big((m-m_1-1)\pi^*(K_X)-S-E'\big)|_S$ represents an effective, nef and big $\bQ$-divisor as $m>m_0+m_1+1$.

When $\zeta>1$, by definition, $|M_{m_0}|$ is composed of a pencil.  Fix, from the very beginning,  a representing effective Weil divisor $K_1\sim K_X$ and set $T_1=\pi^*(K_1)$.
Denote by $T_{1,h}$ the horizontal part of $T_1$. Then $m_0T_{1,h}=E'_{h}$, the horizontal part of $E'$.  Note that ${E'}|_S={E'_{h}}|_S$. Thus Relation (\ref{yx1}) (replacing $\pi^*(K_X)$ with $\pi^*(K_1)$) still holds and $\big((m-m_1-1)\pi^*(K_1)-S-\frac{1}{\zeta}E'\big)|_S$ represents an effective, nef and big $\bQ$-divisor as $m>\frac{m_0}{\zeta}+m_1+1$.

Now we only need to consider the case when $|G|$ is composed of an irrational pencil. Pick two generic irreducible elements $C$ and $C'$ of $|G|$. The vanishing theorem gives the surjective map:
\begin{eqnarray*}
&&H^0(K_S+\roundup{\big((m-m_1-1)\pi^*(K_1)-S-\frac{1}{\zeta}E'\big)|_S}+G)\\
&\lrw& H^0(C, K_C+\tilde{D})\oplus H^0(C', K_{C'}+\tilde{D}')
\end{eqnarray*}
where $\tilde{D}=\roundup{\big((m-m_1-1)\pi^*(K_1)-S-\frac{1}{\zeta}E'\big)|_S}|_C+(G-C)|_C$ has positive degree and so does $\tilde{D}'$.  This implies that the two groups
$H^0(C, K_C+\tilde{D})$ and $H^0(C', K_{C'}+\tilde{D}')$ are non-zero. We are done.
\end{proof}

\subsection{Two restriction maps on canonical class of $(1,2)$-surfaces}\label{3.2}

Within this subsection, we always work under the following assumption:

{\bf ($\mathcal{L}$)} {\em Keep the setting in \ref{setup}. Let $m_1>m_0$ be an integer. Assume that $|M_{m_1}|$ is base point free,  $d_{m_0}=1$, $\Gamma\cong \bP^1$ and that $F$ is a $(1,2)$-surface. Take $|G|=\Mov|K_F|$. Modulo possibly a further birational modification of $X'$, we may assume that $|G|$ is base point free. Let $C$ be a generic irreducible element of $|G|$. } 

\begin{defn}\label{twomaps} For any integers $j\geq 0$, define the following restriction maps:
$$H^0(X', M_{m_1}-jF)\overset{\theta_{m_1,-j}}\lrw H^0(F, M_{m_1}|_F),$$
$$H^0(F, M_{m_1}|_F-jC)\overset{\psi_{m_1,-j}}\lrw H^0(C, M_{m_1}|_C).$$
Set $U_{m_1,-j}=\text{Im}(\theta_{m_1,-j})$, $V_{m_1,-j}=\text{Im}(\psi_{m_1, -j})$,
$u_{m_1, -j}=\dim U_{m_1,-j}$ and $v_{m_1, -j}=\dim V_{m_1, -j}$.
\end{defn}

\begin{prop}\label{-jl} Let $X$ be a minimal projective 3-fold of general type with $P_{m_0}\geq 2$. Keep Assumption ($\mathcal{L}$). 
Then
\begin{itemize}
\item[(1)]  the sequence $\{u_{m_1, -j}|j=0,1,\cdots\}$  is decreasing and so is the sequence $\{v_{m_1,-j}| j=0,1, \cdots\}$.

\item[(2)] We have
$$v_{m_1,0}\leq \begin{cases}
m_1-1, & m_1>2;\\
2,& m_1=2\end{cases}$$

\item[(3)]  If there is a positive integer $k_1$ (resp. $k_2$) such that $P_{m_1}>k_1 u_{m_1,0}$ (resp. $h^0(F, M_{m_1}|_F)>k_2v_{m_1,0}$), 
then
$$M_{m_1}\geq k_1F
\ \ (\text{resp.} M_{m_1}|_F\geq k_2C). $$
\end{itemize}
\end{prop}
\begin{proof}  (1) For any $j\geq 0$, since $M_{m_1}-jF\geq M_{m_1}-(j+1)F$ and
$M_{m_1}|_F-jC\geq M_{m_1}|_F-(j+1)C$, the sequences $\{u_{m_1, -j}\}$ and $\{v_{m_1,-j}\}$ are naturally decreasing. 

(2) Since $\pi^*(K_X)|_F$ is nef and big, we have $\pi^*(K_X)|_F\leq \sigma^*(K_{F_0})\le K_F$ by considering the Zariski decomposition of $K_F$. Hence
$$ (M_{m_1}|_F\cdot C)\leq m_1(\pi^*(K_X)|_F\cdot C)\leq m_1(\sigma^*(K_{F_0})\cdot C)=m_1.$$
Note that $C$ is a curve of genus $2$. If $h^1(C, M_{m_1}|_C)=0$, by Riemann-Roch theorem, we have
$$
h^0(C, M_{m_1}|_C)=(M_{m_1}\cdot C)-1.
$$
Thus we deduce that $v_{m_1,0}\le h^0(C, M_{m_1}|_C)\le m_1-1$, which implies the required inequalities. If $h^1(C, M_{m_1}|_C)>0$,  we have
$$
v_{m_1,0}\le h^0(C, M_{m}|_C)\le 1+\frac{1}{2}(M_{m_1}\cdot C)\le 1+\frac{1}{2}m_1,
$$
where the second inequality follows by Clifford's theorem. We can easily deduce the required upper bound of $v_{m_1,0}$ from the above inequality. The proof of (2) is completed.

(3)  Since $h^0(X', M_{m_1}-k_1F)>0$ by the decreasing property of $\{u_{m_1,-j}\}$, we see $M_{m_1}\geq k_1F$. Similarly, one has $M_{m_1}|_F\geq k_2C.$
\end{proof}

\begin{prop}\label{k1} Let $X$ be a minimal projective 3-fold of general type with $P_{m_0}\geq 2$. Keep Assumption ($\mathcal{L}$).
Suppose that the following conditions hold for some positive integers $n_1$, $j_1$ and $l_1$:
\begin{itemize}
\item[(i)] there exists an effective divisor  $S_{n_1, -j_1}$ on $X'$ such that $|S_{n_1,-j_1}|$ is base point free;
\item[(ii)] $n_1K_{X'}\geq j_1F+S_{n_1, -j_1}$;
\item[(iii)] $S_{n_1,-j_1}|_F\geq l_1\sigma^*(K_{F_0})$ (resp. $S_{n_1,-j_1}|_F\geq l_1C$);
\end{itemize}
Then one has
$$\pi^*(K_X)|_F\geq \frac{l_1+j_1}{n_1+j_1}\sigma^*(K_{F_0})\ \ 
(\text{resp. }\  \pi^*(K_X)|_F\geq \frac{l_1+j_1}{n_1+j_1}C).$$
\end{prop}
\begin{proof} By assumption, we may assume that $S_{n_1,-j_1}$ is smooth. Take a sufficiently large positive integer $s$. Denote by $|N_{sj_1-1}|$
the moving part of $|(sj_1-1)(K_{X'}+F)|$.  By Theorem \ref{KaE}, we have
$$|(sj_1-1)(K_{X'}+F)||_F=|(sj_1-1)K_F|.$$
Since $|(sj_1-1)(K_{X'}+F)|$ clearly distinguishes different fibers of $f$, $|N_{sj_1-1}|$ is big. Modulo a further birational modification, we may and do assume that $|N_{sj_1-1}|$ is base point free. In particular,  $N_{sj_1-1}$ is nef and big. Kawamata-Viehweg vanishing  theorem gives
\begin{eqnarray*}
|s(n_1+j_1)K_{X'}||_F
&\lsgeq &|K_{X'}+N_{sj_1-1}+sS_{n_1,-j_1}+F||_F\\
&=&|K_F+N_{sj_1-1}|_F+sS_{n_1,-j_1}|_F|\\
&\lsgeq&|s(l_1+j_1)\sigma^*(K_{F_0})|.
\end{eqnarray*}
Thus, by the base point free theorem for surfaces, one has
$$\pi^*(K_X)|_F\geq \frac{l_1+j_1}{n_1+j_1}\sigma^*(K_{F_0}).$$
The other statement trivially follows since $K_F\geq \sigma^*(K_{F_0})\geq C$.
\end{proof}

\begin{prop}\label{X1} Let $X$ be a minimal projective 3-fold of general type with $p_g(X)>0$, $P_{m_0}\geq 2$.  Keep Assumption ($\mathcal{L}$).  
 Assume that $|S_1|$ is a moving linear system on $X'$ so that $|S_1|$ and $|F|$ are not composed of the same pencil and that
$$M_{m_1}\geq F+S_1.$$
Suppose that $m_2$ is a non-negative integer such that 
\begin{enumerate}
	\item[(a)] $|(m_2+1)K_{X'}|$ distinguishes different generic irreducible elements of $|M_{m_0}|$;
	\item[(b)] $|(m_2+1)K_{X'}||_{F}$ distinguishes different generic irreducible elements of $|G|$.
\end{enumerate}
Then
\begin{itemize}
\item[(1)] {{Suppose that $|S_1|_F|$ and $|G|$ are not composed of the same pencil. Set $\delta=(S_1|_F\cdot C)$. The following statements hold:}}
    \begin{itemize}
    \item[(1.1)] {{For any positive integer $n>m_1+\frac{1}{\beta}$, one has
    $$(n+1)\xi\ge 2+\delta+\roundup{(n-m_1-\frac{1}{\beta})\cdot \xi}.$$ Moreover, when $S_1|_F$ is big, the above inequality holds for $n\ge m_1+\frac{1}{\beta}$.}}
    \item[(1.2)] {{$\varphi_{n+1, X}$ is birational for all 
    $$
    n\ge \mathrm{max}\{m_2, \llcorner \frac{1}{\beta}+\frac{2m_1}{\delta}+\frac{1}{\mu}(1-\frac{2}{\delta}) \lrcorner+1\}.
    $$
}}
\end{itemize}
\item[(2)] if $|S_1|_F|$ and $|G|$ are composed of the same pencil, the following statements hold:
\begin{itemize}
\item[(2.1)] one has
$$(n+m_1+1)\xi(m_0, |G|)\geq \roundup{n\xi(m_0, |G|)}+2$$
 for any integer $n$ satisfying $n\xi(m_0, |G|)>1$. In particular,
 $$\xi(m_0, |G|)\geq \frac{2}{m_1+1};$$
\item[(2.2)]  $\varphi_{n+m_1+1, X}$ is birational for any integer $n$ satisfying $$n\xi(m_0, |G|)>2 \quad \text{and} \quad n\ge m_2-m_1.$$ \end{itemize}
\end{itemize}
\end{prop}
\begin{proof} Modulo further birational modifications, we may and do assume that $|S_1|$ is base point free.  Let $|G_1|=|S_1|_F|$ and $C_1$ the generic irreducible element of $|G_1|$.
By assumption, $|G_1|$ is also base point free. By the Kawamata-Viehweg vanishing theorem, we have
\begin{eqnarray}
|(n+m_1+1)K_{X'}||_F&\lsgeq&|K_{X'}+\roundup{n\pi^*(K_X)}+S_1+F||_F\notag\\
&\lsgeq& |K_F+\roundup{n\pi^*(K_X)|_F}+C_1|.\label{m1b}
\end{eqnarray}

Since $p_g(X)>0$, we see that $|(n+m_1+1)K_{X'}|$ distinguishes different general $F$ and
$|(n+m_1+1)K_{X'}||_F$ distinguishes different general $C$. What we need to do is to investigate  $|(n+m_1+1)K_{X'}||_C$.
\medskip

(1).
If $|G_1|$ and $|G|$ are not composed of the same pencil, then
$$\xi(m_0, |G|)\geq \frac{1}{m_1}(M_{m_1}|_F\cdot C)\geq \frac{1}{m_1}(C_1\cdot C)\geq \frac{2}{m_1}.$$
We have
$$\frac{1}{\beta}\pi^*(K_X)|_F\equiv C+H_{m_0}$$
where $H_{m_0}$ is certain effective $\bQ$-divisor.  The vanishing theorem on $F$ gives
\begin{eqnarray}
 |K_F+\roundup{n\pi^*(K_X)|_F}+C_1||_C&\lsgeq& |K_F+\roundup{n\pi^*(K_X)|_F-H_{m_0}}+C_1||_C\notag\\
 &=&|K_C+\tilde{D}_{m_0}|\label{m02}
 \end{eqnarray}
where $\deg(\tilde{D}_{m_0})\geq (n-\frac{1}{\beta})\xi+\delta>2$ whenever $n>\frac{1}{\beta}$. Thus $(1.1)$ holds.

For $(1.2)$, set $$n= \mathrm{max}\{m_2, \llcorner \frac{1}{\beta}+\frac{2m_1}{\delta}+\frac{1}{\mu}(1-\frac{2}{\delta}) \lrcorner+1\}.$$ 
Write
$$
m_1\pi^*(K_X)\equiv F+S_1+E_{m_1}.
$$
By the Kawamata-Viehweg vanishing theorem, we have
\begin{align*}
|(n+1)K_{X'}||_F &\lsgeq |K_{X'}+\roundup{n\pi^*(K_X)-\frac{2}{\delta}E_{m_1}-\frac{1}{\mu}\cdot (1-\frac{2}{\delta})E_F^{'}}||_F \notag \\
&\lsgeq |K_F+\roundup{n\pi^*(K_X)|_F-\frac{2}{\delta}E_{m_1}|_F-\frac{1}{\mu}\cdot (1-\frac{2}{\delta})E_F^{'}|_F}|
\end{align*}
since $$n\pi^*(K_X)|_F-\frac{2}{\delta}E_{m_1}|_F-\frac{1}{\mu}\cdot (1-\frac{2}{\delta})E_F^{'}|_F\equiv (n-\frac{2m_1}{\delta}-\frac{1}{\mu}(1-\frac{2}{\delta})\pi^*(K_X)+\frac{2}{\delta}S_1$$ is simple normal crossing (by our assumption), nef and big. The vanishing theorem on $F$ gives
\begin{align*}
&\ \ \ |K_F+\roundup{n\pi^*(K_X)|_F-\frac{2}{\delta}E_{m_1}|_F-\frac{1}{\mu}\cdot (1-\frac{2}{\delta})E_F^{'}|_F}||_C \\
&\lsgeq |K_F+\roundup{n\pi^*(K_X)|_F-\frac{2}{\delta}E_{m_1}|_F-\frac{1}{\mu}\cdot (1-\frac{2}{\delta})E_F^{'}|_F-H_{m_0}}||_C \\
&=|K_C+\tilde{D_n}|
\end{align*}
where $\tilde{D_n}=\roundup{n\pi^*(K_X)|_F-\frac{2}{\delta}E_{m_1}|_F-\frac{1}{\mu}\cdot (1-\frac{2}{\delta})E_F^{'}|_F}-H_{m_0}||_C$ with $$\mathrm{deg}\tilde{D_n}\ge(n-\frac{1}{\beta}-\frac{2m_1}{\delta}-\frac{1}{\mu}(1-\frac{2}{\delta}))\xi+2>2.$$ Thus $\varphi_{\llcorner \frac{1}{\beta}+\frac{2m_1}{\delta}+\frac{1}{\mu}(1-\frac{2}{\delta}) \lrcorner+2, X}$ is birational.

(2) If $|G_1|$ and $|G|$ are composed of the same pencil, then $C_1\equiv C$. By the Kawamata-Viehweg vanishing theorem, we have
\begin{equation}
|K_F+\roundup{n\pi^*(K_X)|_F}+C||_C=|K_C+D_n|\label{m1b2}
\end{equation}
where $\deg(D_n)=\deg(\roundup{n\pi^*(K_X)|_F}|_C)\geq n\xi$.
Whenever $n$ is large enough so that $\deg(D_n)>1$, the base point freeness theorem and Relations (\ref{m1b}) and (\ref{m1b2}) imply that
$$(n+m_1+1)\xi(m_0, |G|)\geq \roundup{n\xi(m_0, |G|)}+2$$
which  also directly implies $\xi(m_0, |G|)\geq \frac{2}{m_1+1}$.
Furthermore, whenever $\deg(D_n)>2$, we see that $\varphi_{n+m_1+1,X}$ is birational.
\end{proof}

\begin{prop}\label{X2} Let $X$ be a minimal projective 3-fold of general type with $p_g(X)>0$, $P_{m_0}\geq 2$.  Keep Assumption ($\mathcal{L}$).  Suppose that $M_{m_1}|_F\geq jC+C_1$ where $C_1$ is an irreducible moving curve on $F$ with $C_1\not\equiv C$ and $j>0$ an integer. Set $\delta_1=(C_1\cdot C)$. Suppose that $m_2$ is the smallest non-negative integer such that 
	\begin{enumerate}
		\item $|(m_2+1)K_{X'}|$ distinguishes different generic irreducible elements of $|M_{m_0}|$;
		\item $|(m_2+1)K_{X'}||_{F}$ distinguishes different generic irreducible elements of $|G|$.
	\end{enumerate}
	 Then
\begin{itemize}
\item[(i)] when $\delta_1\leq 2j$,
$\varphi_{n+1,X}$ is birational for all 
$$
n\ge \mathrm{max}\{m_2, \rounddown{\frac{1}{\xi(m_0, |G|)}(2-\frac{\delta_1}{j})+\frac{1}{\mu}+\frac{m_1}{j}}+1\}.
$$
\item[(ii)] when $\delta_1>2j$,
$\varphi_{n+1,X}$ is birational for all
$$
n\ge \mathrm{max}\{m_2, \rounddown{\frac{1}{\mu}+\frac{2m_1}{\delta_1}+\frac{1}{\beta}(1-\frac{2j}{\delta_1})}+1\}
$$
\item[(iii)] For any positive integer $n$ satisfying $n>\frac{1}{\mu}+\frac{m_1}{j}$ and
$$(n-\frac{1}{\mu}-\frac{m_1}{j})\xi(m_0,|G|)+\frac{\delta_1}{j}>1,$$
one has
$$(n+1)\xi(m_0, |G|)\geq \roundup{(n-\frac{1}{\mu}-\frac{m_1}{j})\xi(m_0,|G|)+\frac{\delta_1}{j}}+2.$$
\end{itemize}
\end{prop}
\begin{proof} Modulo further birational modifications, we may and do assume that $|M_{m_1}|$ is base point free.  By our assumption we may find two effective $\bQ$-divisors
$E_{m_1}'$ on $X'$ and $E_{m_1}''$ on $F$ such that
$$m_1\pi^*(K_X)\equiv M_{m_1}+E_{m_1}',$$
$$M_{m_1}|_F\equiv jC+C_1+E_{m_1}''.$$
Without lose of generality, we may assume that $\mu$ is rational.
Set $$n=\begin{cases}
 \mathrm{max}\{m_2, \rounddown{\frac{1}{\xi(m_0, |G|)}(2-\frac{\delta_1}{j})+\frac{1}{\mu}+\frac{m_1}{j}}+1\}, &\text{when}\  \delta_1\leq 2j;\\
\mathrm{max}\{m_2, \rounddown{\frac{1}{\mu}+\frac{2m_1}{\delta_1}+\frac{1}{\beta}(1-\frac{2j}{\delta_1})}+1\}, &\text{when}\  \delta_1>2j.
\end{cases}$$
For Item (i), since $n\pi^*(K_X)-F-\frac{1}{\mu}E_F'$ is nef and big (see \eqref{muS}, as $F=S$), the Kawamata-Viehweg vanishing theorem implies:
\begin{eqnarray}
|(n+1)K_{X'}||_F&\lsgeq& |K_{X'}+\roundup{n\pi^*(K_X)-\frac{1}{\mu}E_F'}||_F \notag\\
&=& |K_F+\roundup{n\pi^*(K_X)-\frac{1}{\mu}E_F'}|_F| \notag\\
&\lsgeq & |K_F+\roundup{Q_{m_0,m_1}}|
\label{x21}
\end{eqnarray}
where  \begin{eqnarray*}
Q_{m_0,m_1}&=&(n\pi^*(K_X)-\frac{1}{\mu}E_F')|_F-
\frac{1}{j}E'_{m_1}|_F-\frac{1}{j}E''_{m_1}\\
&\equiv& (n-\frac{1}{\mu}-\frac{m_1}{j})\pi^*(K_X)|_F+\frac{1}{j}C_1+C.
\end{eqnarray*}
 Since $Q_{m_0,m_1}-C$ is nef and big, the vanishing theorem implies
\begin{equation}
|K_F+\roundup{Q_{m_0,m_1}}||_C=|K_C+\roundup{Q_{m_0,m_1}-C}|_C|\label{x22}
\end{equation}
where
$$\deg(\roundup{Q_{m_0,m_1}-C}|_C\geq (n-\frac{1}{\mu}-\frac{m_1}{j})\xi(m_0,|G|)+\frac{\delta_1}{j}>2.$$
Clearly, since $p_g(X)>0$,  $|(n+1)K_{X'}|$ distinguishes deferent general $F$ and $|(n+1)K_{X'}||_F$ distinguishes different generic $C$. Combining both (\ref{x21}) and (\ref{x22}), we deduce the  birationality of $\varphi_{n+1,X}$.

Item (iii) follows directly from \eqref{x21} and \eqref{x22} since $|K_C+\roundup{Q_{m_0,m_1}-C}|_C|$ is base point free under the assumption.

We are left to treat (ii).   Since $n\pi^*(K_X)-F-\frac{1}{\mu}E_F'$ is nef and big (see \eqref{muS}, as $F=S$), the Kawamata-Viehweg vanishing theorem implies:
\begin{eqnarray}
|(n+1)K_{X'}||_F&\lsgeq& |K_{X'}+\roundup{n\pi^*(K_X)-\frac{1}{\mu}E_F'}||_F \notag\\
&=& |K_F+\roundup{n\pi^*(K_X)-\frac{1}{\mu}E_F'}|_F| \notag\\
&\lsgeq & |K_F+\roundup{\tilde{Q}_{m_0,m_1}}|
\label{x211}
\end{eqnarray}
where  \begin{eqnarray*}
\tilde{Q}_{m_0,m_1}&=&(n\pi^*(K_X)-\frac{1}{\mu}E_F')|_F-
\frac{2}{\delta_1}E'_{m_1}|_F-\frac{2}{\delta_1}E''_{m_1}-(1-\frac{2j}{\delta_1})H_{m_0}\\
&\equiv& (n-\frac{1}{\mu}-\frac{2m_1}{\delta_1}-\frac{1}{\beta}\cdot(1-\frac{2j}{\delta_1}))\pi^*(K_X)|_F+\frac{2}{\delta_1}C_1+C.
\end{eqnarray*}
 Since $\tilde{Q}_{m_0,m_1}-C$ is nef and big, the vanishing theorem implies
\begin{equation}
|K_F+\roundup{\tilde{Q}_{m_0,m_1}}||_C=|K_C+\roundup{\tilde{Q}_{m_0,m_1}-C}|_C|\label{x222}
\end{equation}
where
$$\deg(\roundup{\tilde{Q}_{m_0,m_1}-C}|_C\geq (n-\frac{1}{\mu}-\frac{m_1}{j})\xi(m_0,|G|)+\frac{\delta_1}{j}>2.$$
Clearly, since $p_g(X)>0$,  $|(n+1)K_{X'}|$ distinguishes deferent general $F$ and $|(n+1)K_{X'}||_F$ distinguishes different generic $C$. Combining (\ref{x211}) and (\ref{x222}), we get the birationality of $\varphi_{n+1,X}$.
\end{proof}

\section{\bf Minimal 3-folds of general type with $p_g=1$}

In this section, we always assume that $p_g(X)=1$.  By the proof of \cite[Corollary 4.10]{EXP3}, we know that $X$ belongs to either of the types: (1) $P_4=1$ and $P_5\geq 3$; (2) $P_4\geq 2$.

\subsection{The case $P_4=1$ and $P_5\geq 3$}\

 As explained in Subsection \ref{basket}, we will utilize those formulae and inequalities in \cite[Section 3]{EXP1} to classify the weighted basket $\mathbb{B}(X)$.  

\begin{prop} If $p_g(X)=P_4(X)=1$ and $|5K_X|$ is composed of a pencil, then $\varphi_{15,X}$ is birational.
\end{prop}
\begin{proof} We may take $m_0=5$ and use the set up in \ref{setup}.  Pick a general fiber $F$ of the induced fibration $f:X'\lrw \Gamma$ from $\varphi_{5}$.  Clearly we have $p_g(F)>0$ and $\zeta\geq P_5(X)-1\geq 2$.  By \eqref{cri}, we have
\begin{equation}\pi^*(K_X)|_F\sim_{\bQ} \frac{\zeta}{\zeta+5}\sigma^*(K_{F_0})+E''_F\label{12i}\end{equation}
for an effective $\bQ$-divisor $E''_F$ on $F$ where $F_0$ is the minimal model of $F$.

For a positive integer $m\geq 7$,  Lemma \ref{s1} says that $|mK_{X'}|$ distinguishes different general $F$.  By Kawamata-Viehweg vanishing theorem, we have
\begin{eqnarray}
|mK_{X'}||_F&\lsgeq& |K_{X'}+\roundup{(m-1)\pi^*(K_X)-\frac{1}{\zeta}E'}||_F\notag\\
&\lsgeq&|K_F+\roundup{(m-1)\pi^*(K_X)|_F-\frac{1}{\zeta}E'|_F}|.\label{p412}
\end{eqnarray}
Noting that $|M_5|$ is composed of a pencil, we have
$$(m-1)\pi^*(K_X)|_F-\frac{1}{\zeta}E'|_F\equiv (m-1-\frac{m_0}{\zeta})\pi^*(K_X)|_F.$$
\medskip

{\bf Case 1}.  $F$ is a not a $(1,2)$-surface.

We have
\begin{eqnarray*}
&&\big((m-1)\pi^*(K_X)|_F-\frac{1}{\zeta}E'|_F\big)-\frac{3\zeta+15}{\zeta}E''_F\notag\\
&\equiv& 3\sigma^*(K_{F_0})+a_{m,\zeta}\pi^*(K_X)|_F
\end{eqnarray*}
where $a_{m,\zeta}=m-1-\frac{m_0+3\zeta+15}{\zeta}>0$ whenever $m\geq 15$.
By Lamma \ref{ll1} and Lemma \ref{ll3}, we see that
$$|K_F+\roundup{(m-1)\pi^*(K_X)|_F-\frac{1}{\zeta}E'|_F-\frac{3\zeta+15}{\zeta}E''_F}|$$
gives a birational map.  Hence we have proved that $\varphi_{15}$ is birational onto its image.
\medskip

{\bf Case 2}.  $F$ is a $(1,2)$-surface.

We take $|G|=\Mov|K_F|$.  We have $\beta\geq \frac{2}{7}$ and $\xi\geq \frac{2}{7}(\sigma^*(K_{F_0})\cdot C)=\frac{2}{7}$ by (\ref{12i}).  By Lemma \ref{s2}(1), when $m\geq 11$, $|mK_{X'}||_F$ distinguishes different generic irreducible elements of $|G|$. Since
$$\alpha(15)=(14-\frac{m_0}{\zeta}-\frac{1}{\beta})\xi\geq \frac{16}{7}>2, $$
$\varphi_{15}$ is birational by Theorem \ref{key-birat}.
\end{proof}

Now we discuss the case when $|5K_X|$ is not composed of a pencil.
\medskip

\noindent {\bf Setting ($\aleph$-1).}
{\it Take two different general members $S_5$,  $S_5'\in |M_5|$. Denote by $\Lambda_5$ the 1-dimensional sub-pencil, of $|M_5|$, generated by $S_5$ and $S_5'$.  Modulo a further birational modification, we may and do assume that both $|M_5|$ and the moving part of $\Lambda_5$ are base point free. Then one gets an induced fibration $f=f_{\Lambda_5}:X'\lrw \bP^1$ whose general fiber is denoted as $F$, which has the same birational invariants as that of a general member of $|M_5|$.  In particular, $p_g(F)=p_g(S_5)\geq 2$.
We may take $m_0=5$ and $|G|=|M_5|_F|$. Pick a generic irreducible element $C_5$ in $|G|$. Clearly $\beta\geq \frac{1}{5}$.  On the other hand, we have
\begin{equation}\pi^*(K_X)|_F\geq \frac{1}{6}\sigma^*(K_{F_0})\label{ib}\end{equation}
by \eqref{cri} as $\mu\geq \frac{1}{5}$.
}

\begin{prop}\label{C5} Assume that $p_g(X)=P_4(X)=1$ and that $|5K_X|$ is not composed of a pencil. Keep the setting in {\bf ($\aleph$-1)}. If $g(C_5)\geq 3$, then $\varphi_{16,X}$ is birational.
\end{prop}
\begin{proof} We have $m_0=5$, $\Lambda_5\subset |5K_{X'}|$, $\zeta=1$ and $\beta\geq \frac{1}{5}$. Since $g(C_5)\geq 3$, we have
$\xi\geq \frac{4}{11}$ by Subsection \ref{var} and Inequality (\ref{kieq2}).  Take $m=14$. Then, since $\alpha(14)\geq \frac{12}{11}>1$, one has $\xi\geq \frac{3}{7}$ by Inequality (\ref{kieq1}). Finally, since $\alpha(16)\geq \frac{15}{7}>2$,  $\varphi_{16,X}$ is birational by Lemma \ref{s1}, Lemma \ref{s2} and Theorem \ref{key-birat}.
\end{proof}

\begin{prop}\label{C52}  Assume that $p_g(X)=P_4(X)=1$ and that $|5K_X|$ is not composed of a pencil. Keep the setting in ($\aleph$-1). If $g(C_5)=2$, then $\varphi_{16,X}$ is birational.
\end{prop}
\begin{proof}  By \cite[(3.6)]{EXP1}, $n_{1,4}^0\geq 0$ implies that $\chi(\OO_X)\geq P_5\geq 3$.   We will discuss the two cases separately: $q(F)=0$ or $q(F)>0$.
\medskip

{\bf Case 1.} $q(F)=0$.

With the fibration $f:X'\lrw \bP^1$,
we have $q(X)=q(X')=0$ and $h^2(\OO_X)=\chi(\OO_X)$. Since $q(F)=0$, one has 
$R^1f_*\omega_{X'}=0$ and hence
$$h^1(\bP^1, f_*\omega_{X'})=h^2(\OO_{X'})=\chi(\OO_X)\geq 3.$$
Since $f_*\omega_{X'/\mathbb{P}^1}$ is a nef vector bundle of rank $p_g(F)$, we may write
$$
f_*\omega_{X'}=\oplus_{i=1}^{p_g(F)}\mathcal{O}_{\mathbb{P}^1}(a_i),
$$
where $a_i\ge -2$ for any $1\le i\le p_g(F)$. Since $p_g(X)=1$, there is an $i_0$ such that $a_{i_0}=0$. We deduce that 
$$
h^1(\mathbb{P}^1, f_*\omega_{X'})=\sum_{i\neq i_0}h^1(\mathbb{P}^1,\mathcal{O}_{\mathbb{P}^1}(a_i))\le p_g(F)-1.
$$
 Thus we have $p_g(F)\ge 4$.

\noindent{\bf Subcase (1-i)}. $|K_F|$ is not composed of a pencil.

We consider the natural restriction map
$$\theta: H^0(F, \sigma^*(K_{F_0}))\lrw H^0(C_5, \sigma^*(K_{F_0})|_{C_5}).$$

When $\dim_k(\text{Im}(\theta))\geq 3$, then we have $\deg(\sigma^*(K_{F_0})|_{C_5})\geq 4$ by
Riemann-Roch and the Clifford theorem on $C_5$. Hence
$$\xi\geq \frac{1}{6}(\sigma^*(K_{F_0})\cdot C_5)\geq \frac{2}{3}.$$
Since $\alpha(15)\geq \frac{8}{3}>2$, $\varphi_{15,X}$ is birational by Lemma \ref{s1}, Lemma \ref{s2} and Theorem \ref{key-birat}.

When $\dim_k(\text{Im}(\theta))\leq 2$, we naturally have
$$\sigma^*(K_{F_0})\geq C_5+C'$$
where $C'$ is a generic irreducible element in $\Mov|\sigma^*(K_{F_0})-C_5|$. Suppose $C_5\equiv C'$. Then $\pi^*(K_X)|_F\geq \frac{1}{3}C_5$
which means $\beta\geq \frac{1}{3}$.  Since $$\xi\geq \frac{1}{6}(\sigma^*(K_{F_0})\cdot C_5)\geq \frac{1}{3}$$ by Lemma \ref{ll2}, we have
$\alpha(16)\geq \frac{7}{3}>2.$
Hence $\varphi_{16,X}$ is birational for the similar reason. Suppose that $C_5$ and $C'$ are not in the same curve family, in particular, $C_5\not\equiv C'$. Then $(C_5\cdot C')\geq 2$ since $|C_5|$ is moving on $F$. By the vanishing theorem, we have
\begin{eqnarray*}
|13K_{X'}||_F
&\lsgeq& |K_{X'}+\roundup{7\pi^*(K_X)}+F||_F\\
&\lsgeq&|K_F+\roundup{Q_6}+C'+C_5|
\end{eqnarray*}
where $Q_6\equiv \pi^*(K_X)|_F$ is nef and big.  By Lemma \ref{s1} and Lamma \ref{s2}, $|13K_{X'}|$ distinguishes different general fiber $F$ and different generic elements $C_5$. Using the vanishing theorem once more, we have
$$|K_F+\roundup{Q_6}+C'+C_5||_{C_5}=|K_{C_5}+D_5|$$
where $\deg(D_5)=((\roundup{Q_6}+C')\cdot C_5)>2$. Thus $\varphi_{13,X}$ is birational.
\medskip

\noindent{\bf Subcase (1-ii)}.  $|K_F|$ is composed of a pencil.

Modulo further birational modifications, we may and do assume that  $\Mov|K_F|$ is base point free. Since $q(F)=0$, $\Mov|K_F|$ is composed of a rational pencil. Let $\tilde{C}$ be a generic irreducible element of $\Mov|K_F|$. Since $p_g(F)\ge 4$, we have $\sigma^*(K_{F_0})\ge 3\tilde{C}$.

If $C_5$ is not numerically equivalent to $\tilde{C}$, then $h^0(F, \tilde{C}-C_5)=0$ and then 
$h^0(C_5, \tilde{C}|_{C_5})\geq 2$. We have $(\tilde{C}\cdot C_5)\ge 2$ by the Riemann-Roch on $C_5$. Thus we have
$$\xi\geq \frac{1}{6}(\sigma^*(K_{F_0})\cdot C_5)\geq \frac{1}{2}(C\cdot C_5)\geq 1.$$
In this case we have seen that $\beta\geq \frac{1}{5}$.

Otherwise, we have $\sigma^*(K_{F_0})\geq 3C_5$ and so $\beta\geq \frac{1}{2}$. Also $\xi\geq
\frac{1}{6}(\sigma^*(K_{F_0})\cdot C_5)\geq \frac{1}{3}$ by Lemma \ref{ll2}.  In both cases, one has
$\alpha(15)\geq \frac{7}{3}>2$. Thus $\varphi_{15,X}$ is birational by Lemma \ref{s1}, Lemma \ref{s2} and Theorem \ref{key-birat}.
\medskip

{\bf Case 2}.  $q(F)>0$. \

By Debarre \cite{D}, one has $K_{F_0}^2\geq 2p_g(F)\geq 4$.  Assume that  $|G|$ is not composed of a pencil.  Then we have
$$\xi\geq \frac{1}{6}(\sigma^*(K_{F_0})\cdot C_5)\geq \frac{1}{2}$$
since   $(\sigma^*(K_{F_0})\cdot C_5)\geq \sqrt{8}>2$.  Then it follows that $\alpha(16)>2$.  When $|G|$ is composed of  an irrational pencil, we have $G\geq 2C_5$ and so
$\beta\geq \frac{2}{5}$. Note that one has $\xi\geq \frac{1}{3}$ and so that  $\alpha(15)>2$.  As a conclusion, for above two cases, $\varphi_{16,X}$ is birational by Lemma \ref{s1}, Lemma \ref{s2} and Theorem \ref{key-birat}.

{}From now on,  we may and do assume that $|G|$ is composed of a rational pencil.
Since $F$ possess  a genus 2 fibration onto $\bP^1$ and $F$ is of general type, we see $q(F)=1$ (see Xiao \cite[Theorem 2.4.10]{X91}).
\medskip

\noindent{\bf Subcase (2-i)}.  $K_{F_0}^2\geq 6$.

As $2=(K_{F}\cdot C_5)\geq \sqrt{K_{F_0}^2\cdot \sigma_*(C_5)^2}$, we see $\sigma_*(C_5)^2=0$.
By \cite[Lemma 2.7]{C14}, we see
\begin{equation}\sigma^*(K_{F_0})\geq \frac{3}{2}C_5\label{ib2}\end{equation}
and so $\beta\geq \frac{1}{4}$ according to (\ref{ib}).

Let us consider the natural map:
$$H^0(F, 2\sigma^*(K_{F_0})-iC_5)\overset{\rho_{-i}}\lrw H^0(C_5,
2\sigma^*(K_{F_0})|_{C_5})$$
for $0\leq i\leq 3$.
Note that $h^0(F,2\sigma^*(K_{F_0}))=P_2(F)\geq 8$ and
$$h^0(C_5,2\sigma^*(K_{F_0})|_{C_5})=h^0(\sigma_*(C_5), 2K_{F_0}|_{\sigma_*(C_5)})=3.$$
We naturally have  $h^0(F, 2\sigma^*(K_{F_0})-C_5)\geq 5.$
\smallskip

\noindent{\bf (2-i-1)}. If $\dim\text{Im}(\rho_{-1})=3$, we have
$$2\sigma^*(K_{F_0})\geq C_5+C_{-1}$$
where $|C_{-1}|=\Mov|2\sigma^*(K_{F_0})-C_5|$ and $(C_{-1}\cdot C_5)=\deg(C_{-1}|_{C_5})\geq 4$. By (\ref{ib}) and (\ref{ib2}), we have
\begin{eqnarray}
8\pi^*(K_X)|_F&\geq&\frac{4}{3}\sigma^*(K_{F_0})\geq \sigma^*(K_{F_0})+\frac{1}{3}\cdot\frac{3}{2}C_5\notag\\
&\geq& C_5+\frac{1}{2}C_{-1}.\label{ib3}
\end{eqnarray}
Applying Kawamata-Viehweg vanishing theorem, one gets
\begin{eqnarray}
|15K_{X'}||_F&\lsgeq&|K_{X'}+\roundup{9\pi^*(K_X)}+F||_F\notag\\
&\lsgeq& |K_F+\roundup{Q_{-1}}+C_5|\label{ib4}
\end{eqnarray}
where  $Q_{-1}\equiv\pi^*(K_X)|_F+\frac{1}{2}C_{-1}$.
Since
$$((\pi^*(K_X)|_F+\frac{1}{2}C_{-1})\cdot C_5)\geq (\pi^*(K_X)|_F\cdot C_5)+\frac{1}{2}(C_{-1}\cdot C_5)>2,$$
we see that $|K_F+\roundup{Q_{-1}}+C_5||_{C_5}$ gives a birational map. By Lemma \ref{s1} and Lemma \ref{s2}, we have seen that $\varphi_{15,X}$ is birational.
\smallskip

\noindent{\bf (2-i-2)}. If $\dim\text{Im}(\rho_{-1})\leq 2$ and $\dim\text{Im}(\rho_{-2})=2$, we have
$$2\sigma^*(K_{F_0})\geq 2C_5+C_{-2}$$
where $|C_{-2}|=\Mov|2\sigma^*(K_{F_0})-2C_5|$ and $(C_{-2}\cdot C_5)\geq 2$.
By the vanishing theorem, we have
\begin{eqnarray*}
|16K_{X'}||_F&\lsgeq& |K_{X'}+\roundup{10\pi^*(K_X)}+F||_F\\
&\lsgeq&|K_F+\roundup{Q_{-2}}+C_5|
\end{eqnarray*}
where $Q_{-2}\equiv 4\pi^*(K_X)|_F+\frac{1}{2}C_{-2}$ is nef and big.  Since $(Q_{-2}\cdot C_5)\geq \frac{7}{3}>2$, we see that $|K_F+\roundup{Q_{-2}}+C_5||_{C_5}$ gives a birational map. Thus $\varphi_{16,X}$ is birational by Lemma \ref{s1} and Lemma \ref{s2}.
\smallskip

\noindent{\bf (2-i-3)}.   If $\dim\text{Im}(\rho_{-1})\leq 2$ and $\dim\text{Im}(\rho_{-2})=1$, we have
$$2\sigma^*(K_{F_0})\geq 4C_5$$
since $h^0(F, 2\sigma^*(K_{F_0})-2C_5)\geq 3$.  Clearly this implies $\beta\geq \frac{1}{3}$ by (\ref{ib}).  Since $\alpha(16)>2$, $\varphi_{16,X}$ is birational by Lemma \ref{s1}, Lemma \ref{s2} and Theorem \ref{key-birat}.
\smallskip

\noindent{\bf Subcase (2-ii)}.  If $K_{F_0}^2\leq 5$,  by Horikawa's theorem (see \cite{H1,H2,H3,H4}), the Albanese map of $F$ is a genus 2 fibration onto an elliptic curve $E$, say $\text{alb}: F\lrw E$.
On the other hand,  $K_{F_0}^2\geq 2p_g(F)$ implies $p_g(F)=2$. Modulo further birational modification, we may and do assume that $\Mov|K_F|$ is base point free. Pick a generic irreducible element $\hat{C}$ of $\Mov|K_F|$. If $\Mov|K_F|$ is composed of an irrational pencil, then $\hat{C}$ and $C_5$ are not in the same pencil, i.e. $(\hat{C}\cdot C_5)\geq 1$. Since $|G|$ is composed with a rational pencil, we have $h^0(\hat{C}, {C_5}|_{\hat{C}})\geq 2$ and then $(\hat{C}\cdot C_5)\ge 2$ by the Riemann-Roch theorem on $\hat{C}$.  
Moreover, numerically, one has $K_F\geq 2\hat{C}$. Then $2=(K_F\cdot C_5)\geq 4$, a contradiction. So $\Mov|K_F|$ must be a rational pencil. Write
$$\sigma^*(K_{F_0})\sim \hat{C}+H$$ where $H$ is an effective divisor. Pick a general fiber $C'$ of $alb$. Clearly we have $(\hat{C}\cdot C')=2$ as $|\hat{C}|$ is a rational pencil and $C'\not\in |\hat{C}|$. Also we have $2\geq (\sigma^*(K_{F_0})\cdot C')\geq 2$ by Lemma \ref{ll2}. Thus $(C'\cdot H)=0$ which means $H$ is vertical with respect to $alb$. So $H^2\leq 0$. Now one has $4\geq (\hat{C}+H)^2=\sigma^*(K_{F_0})^2\geq 4$ since
$(\sigma^*(K_{F_0})\cdot \hat{C})=2$ by Lemma \ref{ll2}. Thus $H^2=0$ and $H$ is equivalent to a multiple of $C'$. The only possibility is $H\equiv C'$.  Now we see that $C_5\sim \hat{C}$, otherwise, $(K_F\cdot C_5)>2$ gives a contradiction.
Hence we have
$$6\pi^*(K_X)|_F\geq \sigma^*(K_{F_0})\equiv C_5+C'$$ with $(C_5\cdot C')=2$.
Applying Kawamata-Viehweg vanishing theorem, one gets
\begin{eqnarray*}
|13K_{X'}||_F&\lsgeq&|K_{X'}+\roundup{7\pi^*(K_X)}+F||_F\\
&\lsgeq& |K_F+\roundup{Q_{-3}}+C_5|
\end{eqnarray*}
where  $Q_{-3}\equiv\pi^*(K_X)|_F+C'$ and $(Q_{-3}\cdot C_5)>2$.
Since $$|K_F+\roundup{Q_{-3}}+C_5||_{C_5}$$ gives a birational map,  we see that $\varphi_{13,X}$ is birational by Lemma \ref{s1} and Lemma \ref{s2},
\end{proof}

Proposition \ref{C5} and Proposition \ref{C52} directly imply the following:
\begin{thm}\label{t1} Let $X$ be a minimal projective 3-fold of general type with $p_g(X)=P_4(X)=1$. Then $\varphi_{16,X}$ is birational onto its image.
\end{thm}

\subsection{The case $P_4\geq 2$}\

\begin{prop}\label{421} Let $X$ be a minimal projective 3-fold of general type with $p_g(X)=1$ and $P_4\geq 2$. Then $\varphi_{16,X}$ is birational unless $P_4(X)=2$ and $|4K_X|$ is composed of a rational pencil of $(1,2)$-surfaces.
\end{prop}
\begin{proof}  Take $m_0=4$. Keep the same notation as in \ref{setup}.  By Theorem \ref{KaE}, we have
\begin{equation} |15K_{X'}||_S\lsgeq |3(K_{X'}+S)||_S=|3K_S|\label{15x}\end{equation}
for a generic irreducible element $S$ of $|M|$.  This also implies that
$\pi^*(K_X)|_S\geq \frac{1}{5}\sigma^*(K_{S_0})$.

When $d_1\geq 2$, we have $\zeta=1$ by definition.  Note that $K_S\sim (K_{X'}+S)|_S$. The uniqueness of Zariski decomposition implies that
$$\sigma^*(K_{S_0})\geq \pi^*(K_X)|_S+S|_S\geq \frac{5}{4}S|_S,$$
which means that $K_{S_0}^2\geq \frac{5}{4}>1$.
Thus $S_0$ is not a $(1,2)$-surface. Take $|G|=|S|_S|$. Then $\beta\geq \frac{1}{4}$.
By Lemma \ref{ll2}, we get $\xi\geq \frac{2}{5}$ and so
$$\alpha(15)\geq (15-1-4-4)\xi \geq \frac{12}{5}>2.$$
By Lemma \ref{s1}, Lemma \ref{s2} and Theorem \ref{key-birat}, $\varphi_{15,X}$ is birational.

When $d_1=1$, using Lemma \ref{s1} and (\ref{15x}), we may and do assume that
$F$ is either a $(2,3)$ or a $(1,2)$-surface.
For the case of $(2,3)$-surfaces,  we take $|G|=\text{Mov}|K_F|$. Then $\beta=\frac{1}{5}$ and we still have $\xi\geq \frac{2}{5}$.  Then, since $\alpha(16)>2$, $\varphi_{16,X}$ is birational by Theorem \ref{key-birat}.  For the case of $(1,2)$-surfaces,  we still take $|G|=\text{Mov}|K_F|$.  If $P_4>2$ or $P_4=2$ and $|M|$ is an irrational pencil, then we have $\zeta\geq 2$. This implies that $\pi^*(K_X)|_F\geq \frac{1}{3}\sigma^*(K_{F_0})$ and $\beta\geq \frac{1}{3}$. Then we have
$\xi\geq \frac{1}{3}$ and 
$\alpha(15)>2$. By Lemma \ref{s1}, Lemma \ref{s2} and Theorem \ref{key-birat}, $\varphi_{15, X}$ is birational. \end{proof}

\noindent{\bf ($\aleph$-2).} {\it Assume $p_g(X)=1$, $P_4(X)=2$ and $|M|$ is composed of a rational pencil of $(1,2)$-surfaces}.  One has $\chi(\OO_{X'})>0$ since $P_3=P_2$.  Furthermore, from the induced fibration $f:X'\rightarrow \bP^1$, one gets $q(X)=0$, $\chi(\OO_X)=h^2(\OO_{X'})=h^1(f_*\omega_{X'})\leq 1$ and, due to
 $n_{1,4}^0\geq 0$ (\cite[(3.6)]{EXP1}),
$3\geq \chi+2\geq P_5+\sigma_5. $

\begin{prop}\label{422} Let $X$ be a minimal projective 3-fold of general type with $p_g(X)=1$ and $P_4(X)=2$. Assume that $X$ has the property {\bf  ($\aleph$-2)}. Then $\varphi_{17,X}$ is birational.
\end{prop}
\begin{proof}  Note that $2\leq P_5\leq 3$. By \cite[Table A3]{EXP3}, we know $\xi\geq \frac{2}{7}$ and $K_X^3\geq \frac{1}{70}$.  \medskip

{\bf Case 1}. $P_5=3$.  If $|5K_{X'}|$ is composed of a pencil, then we have
$5\pi^*(K_X)\geq 2F$,  which means $\mu\geq \frac{2}{5}$. This also implies that $\pi^*(K_X)|_F\geq \frac{2}{7}\sigma^*(K_{F_0})$ whence $\beta\geq \frac{2}{7}$. 
Since $\alpha(15)\geq \frac{16}{7}>2$, $\varphi_{15,X}$ is birational by the similar reason.

Now assume that $|5K_{X'}|$ is not composed of a pencil. Clearly we have $h^0({M_5}|_F)\geq 2$. Set $|G_5|=|{M_5}|_F|$.
\medskip

\noindent{\bf Subcase 1.1}.  When $|G|$ and $|G_5|$ are not composed of the same pencil, one has $\xi\geq \frac{1}{5}(M_5|_F\cdot C)\geq \frac{2}{5}$.  Recall that we have $m_0=4$, $\zeta=1$ and $\beta=\frac{1}{5}$.  So $\alpha(16)>2$ and $\varphi_{16, X}$ is birational.
\medskip

\noindent{\bf Subcase 1.2}.  When $|G|$ and $|M_5|_F|$ are composed of the same pencil, we must have $(C\cdot G_5)=0$.  
Recall that we have $\xi\geq \frac{2}{7}$.

If $M_5|_F$ is not irreducible for a general $M_5$, we have $\beta\geq \frac{2}{5}$. Since $\alpha(15)>2$,  $\varphi_{15,X}$ is birational by Lemma \ref{s1}, Lemma \ref{s2} and Theorem \ref{key-birat}.

If $M_5|_F$ is irreducible for the general $M_5$, we denote this curve by $C_5=M_5\cap F$. On $F$, we have $C\equiv C_5$. Take a generic irreducible element $\hat{C}$ of $|M_5|_{M_5}|$. Suppose $(C_5\cdot \hat{C})>0$. We must have $(C_5\cdot \hat{C})\geq 2$ since $|C_5|$ is a rational pencil. So we get
$$\xi=(\pi^*(K_X)\cdot C_5)=(\pi^*(K_X)|_{M_5}\cdot C_5)\geq \frac{1}{5}(\hat{C}\cdot C_5)\geq \frac{2}{5}.$$
We again have $\alpha(16)>2$, which implies the birationality of $\varphi_{16,X}$.
Suppose $(C_5\cdot \hat{C})=0$ and $M_5|_{M_5}\geq 2\hat{C}$.  We set $\tilde{m}_0=5$. Then $\zeta(\tilde{m}_0)=1$. Set $|\tilde{G}|=|M_5|_{M_5}|$. Clearly we have $\beta(\tilde{m}_0, |\tilde{G}|)=\frac{2}{5}$ and
$$\xi(\tilde{m}_0, |\tilde{G}|)=\xi(m_0, |G|)\geq \frac{2}{7}.$$ Since
$\alpha_{(\tilde{m}_0, |\tilde{G}|)}(16)\geq \frac{15}{7}>2$, $\varphi_{16,X}$ is birational by Lemma \ref{s1}, Lemma \ref{s2} and Theorem \ref{key-birat}.  As the last step, suppose  $(C_5\cdot \hat{C})=0$ and $M_5|_{M_5}$ is irreducible.  Then we have $C_5\equiv M_5|_{S_5}$ on a general member $S_5\in |M_5|$. We know that $S_5$ is not a $(1,2)$-surface and $p_g(S_5)\geq 2$. So $$\xi=\xi(4, |G|)=(\pi^*(K_X)|_{S_5}\cdot M_5|_{S_5})
\geq\frac{1}{6}(\sigma_5^*(K_{S_{5,0}})\cdot \hat{C})\geq \frac{1}{3} $$
where $\sigma_5:S_5\lrw S_{5,0}$ is the contraction onto the minimal model. Thus $\alpha(17)\geq \frac{7}{3}>2$ and $\varphi_{17,X}$ is birational for the similar reason.
\medskip

{\bf Case 2}. $P_5=2$. Since we have $\chi(\OO_X)=1$, $\epsilon_5\geq 0$ implies $P_6+\sigma_5\leq 5$. By $\epsilon_6=0$, we get $P_6=P_7$ and $\epsilon=0$. Hence $\sigma_5=0$.  By \cite[Lemma 3.2]{EXP2}, we have $P_6\geq P_4+P_2=3$.
We set $m_1=6$ and shall use Proposition \ref{X1} to consider in details the property of the maps $\theta_{6,-j}$ and $\psi_{6,-j}$ for $j\geq 0$ (see Definition \ref{twomaps}). Recall that $m_0=4$.

If $v_{6,0}\geq 3$, one has $\xi\geq\frac{1}{6}\deg(M_6|_C)\geq \frac{2}{3}$ by Riemann-Roch and Clifford's theorem on $C$. Since $\alpha(14)\geq \frac{8}{3}>2$, $\varphi_{14,X}$ is birational by Lemma \ref{s1}, Lemma \ref{s2} and Theorem \ref{key-birat}.
If $v_{6,-1}\geq 2$,  then we have
$$M_6|_F\geq C+C_{-1}$$
where $C_{-1}$ is a moving curve with $h^0(C,C_{-1}|_C)\geq 2$ (hence $(C\cdot C_{-1})\geq 2$).
Since Kawamata-Viehweg vanishing theorem gives
$$|12K_{X'}||_F\lsgeq |K_F+\roundup{\pi^*(K_X)|_F}+C_1+C|$$
and $|K_F+\roundup{\pi^*(K_X)|_F}+C_1+C||_C=|K_C+D|$ with $\deg(D)>2$. Thus
$\varphi_{12,X}$ is birational.

We assume, from now on, that $v_{6,0}\leq 2$ and $v_{6,-1}\leq 1$.
\medskip

\noindent{\bf Subcase 2.1}. Either  $h^0(F, M_6|_F)\geq 4$ or $h^0(F, M_6|_F)=3$ and $v_{6,0}=1$.

Clearly,  one has $M_6|_F\geq 2C$, which means $\beta\geq \frac{1}{3}$.  Since $\alpha(16)\geq \frac{16}{7}>2$, $\varphi_{16,X}$ is birational by the similar reason.
\medskip

\noindent{\bf Subcase 2.2}.  Either $h^0(M_6|_F)=3$ and $v_{6,0}\geq 2$ or $h^0(F, M_6|_F)\leq 2$.

In particular, we have $u_{6,0}\leq 3$.
When $P_6\geq 5$ or  $P_6=4$ and $u_{6,0}\leq 2$, one naturally  has
$$M_6\geq F+F_1$$
for a moving divisor $F_1$ on $X'$.  Suppose that $F$ and $F_1$ are in the same algebraic family. Then $\mu\geq \frac{1}{3}$ and hence $\beta\geq \frac{1}{4}$.
 As $\alpha_{(m_0, |G|)}(16)\geq \frac{16}{7}>2$, $\varphi_{16,X}$ is birational by the similar reason. Suppose that $F$ and $F_1$ are in different algebraic families.
By Proposition \ref{X1} ($m_0=4$, $m_1=6$), we see that either $\varphi_{13,X}$ is birational or we can get a better estimate for $\xi$. In fact, since  $\xi\geq \frac{2}{7}$ and if we take $n=4$, Proposition \ref{X1}(2.1) gives $\xi\geq \frac{4}{11}$; similar trick implies $\xi\geq \frac{2}{5}$. Now since $6\xi>2$, we see that $\varphi_{13,X}$ is birational by Proposition \ref{X1}(2.2).

When $P_6=4$ and $u_{6,0}=3$,  then we must have $h^0(M_6|_F)=3$. By assumption,  one gets $v_{6,0}=2$, which implies $\xi\geq \frac{1}{3}$. Since $\alpha(17)>2$,  $\varphi_{17,X}$ is birational.

When $P_6=3$,  we have
$$B^{(5)}=\{7\times (1,2), 2\times (2,5), 2\times (1,3), (1,4)\},\ K^3=\frac{1}{60}.$$
By \cite[Lemma 3.2]{EXP2}, we have $P_8\geq P_6+P_2=4$ and $\epsilon_7=5-P_8\geq 0$.  Since any one-step packing of $B^{(5)}$ has the volume $<\frac{1}{70}$, we see $B_X=B^{(5)}$ and $K_X^3=\frac{1}{60}$. Note that we have $\xi\ge\frac{2}{7}$. Since $r_X=60$ and $r_X\xi$ is an integer, we see $\xi\geq \frac{3}{10}$.  Thus $\alpha(17)>2$ and $\varphi_{17,X}$ is birational.
\end{proof}

\begin{prop}\label{423} Let $X$ be a minimal projective 3-fold of general type with $p_g(X)=1$ and $P_4(X)=2$. Assume that $X$ has the property {\bf  ($\aleph$-2)}.  Then $\varphi_{16,X}$ is birational unless $X$ belongs to one of the following types:
\begin{itemize}
\item[(i)] $B_X=\{4\times (1,2), (3,7), 3\times (2,5), (1,3)\}$,
 $K^3=\frac{2}{105}$;
 \item[(ii)] $B_X=\{4\times (1,2), (5,12), 2\times (2,5), (1,3)\}$,
$K^3=\frac{1}{60}$;

\item[(iii)] $B_X=\{7\times (1,2), (3,7), 2\times (1,3), (2,7)\}$,
$K^3=\frac{1}{42}$;

\item[(iv)] $B_X=\{7\times (1,2), (3,7), (1,3), (3,10)\}$,
$K^3=\frac{2}{105}$;

\item[(v)]  $B_X=\{7\times (1,2), 2\times (2,5), 2\times (1,3), (1,4)\}$, $K^3=\frac{1}{60}$.
\end{itemize}
\end{prop}
\begin{proof} From the proof of the previous proposition, we only need to consider the following three situations:
\begin{itemize}
\item[{\bf (\ref{423}.1)}] $P_5=3$ (the last situation of Subcase 1.2 of Proposition \ref{422});

\item[{\bf (\ref{423}.2)}] $P_5=2$ and $P_6=P_7=4$ (the second situation of Subcase 2.2 of Proposition \ref{422});

\item[{\bf (\ref{423}.3)}] $P_5=2$ and $P_6=P_7=3$ (the last situation of Subcase 2.2 of Proposition \ref{422}).
\end{itemize}

{\bf Step 1}.  Either $P_7\geq 5$ or $P_8\geq 6$.

We keep the same notation as in Proposition \ref{-jl}. Take $m_1=7$. If $v_{7,0}\geq 3$, then $\xi\geq \frac{1}{7}\deg(M_7|_C) \geq \frac{4}{7}$. Since $\alpha(14)\geq \frac{16}{7}>2$, $\varphi_{14,X}$ is birational. If $v_{7,0}\leq 2$ and $u_{7,0}\geq 4$, then $M_7|_F\geq C+C_1$ for certain moving curve $C_1$ (i.e. $h^0(F, C_1)\geq 2$).  For the case $C\equiv C_1$, we have $\beta\geq \frac{2}{7}$ and $\alpha(16)\geq \frac{15}{7}>2$. Hence $\varphi_{16,X}$ is birational. For the case $C\not\equiv C_1$, $ \varphi_{13,X}$ is birational by Proposition \ref{X2}.
If  $v_{7,0}\leq 2$ and $u_{7,0}\leq 3$, since $P_7\geq 5$, $M_7\geq F+F_1$ for some moving divisor $F_1$ on $X'$. In the case $F\equiv F_1$, we have $\mu\geq \frac{2}{7}$ and then $\beta\geq \frac{2}{9}$. Since $\alpha(13)\geq \frac{8}{7}>1$, we see $\xi\geq \frac{4}{13}$ Since $\alpha(16)\geq \frac{28}{13}>2$, $\varphi_{16,X}$ is birational. Finally, for the case $F\not\equiv F_1$,  by Proposition \ref{X1}, either $\varphi_{14,X}$ is birational or we have  that $|S_1|_F|$ and $|G|$ are not composed of the same pencil. Take $n=4$ and run Proposition \ref{X1}(2.1), we get $\xi\geq \frac{1}{3}$. Similarly, take $n=7$, since $7\xi>2$, $\varphi_{15,X}$ is birational again due to Proposition \ref{X1}(2.2).

Similarly, take $m_1=8$.  If $v_{8,0}\geq 3$, then $\xi\geq \frac{1}{8}\deg(M_8|_C) \geq \frac{1}{2}$. Since $\alpha(15)\geq \frac{5}{2}>2$, $\varphi_{15,X}$ is birational. If $v_{8,0}\leq 2$ and $u_{8,0}\geq 4$, then $M_8|_F\geq C+C_1$ for certain moving curve $C_1$ on $F$.  For the case $C\equiv C_1$, we have $\beta\geq \frac{1}{4}$ and the optimization by Inequality (\ref{kieq1}) gives $\xi\geq \frac{4}{13}$.  Hence $\alpha(16)>2$ and $\varphi_{16,X}$ is birational. For the case $C\not\equiv C_1$, $ \varphi_{14,X}$ is birational by Proposition \ref{X2}.
If  $v_{8,0}\leq 2$ and $u_{8,0}\leq 3$, since $P_8\geq 6$,  $h^0(M_8-F)\geq 3$.  In the case that $|M_8-F|$ is composed of the same pencil as $|F|$, we have $M_8\geq 3F$.
Then $\mu\geq \frac{3}{8}$ and $\beta\geq \frac{3}{11}$. As $\xi\geq \frac{2}{7}$ and
 $\alpha(15)\geq \frac{46}{21}>2$, $\varphi_{15,X}$ is birational. Finally, for the case  $|M_8-F|$ is not composed of the same pencil as $|F|$, we may write $M_8\geq F+F_1$ for some moving divisor $F_1$ with $F\not\equiv F_1$.  By Proposition \ref{X1}, either $\varphi_{15,X}$ is birational or we have  that $|F_1|_F|$ and $|G|$ are not composed of the same pencil. Take $n=4$ and run Proposition \ref{X1}(2.1), we get $\xi\geq \frac{4}{13}$. Similarly, take $n=7$, since $7\xi>2$, $\varphi_{16,X}$ is birational due to Proposition \ref{X1}(2.2).

Therefore we may assume $P_7\leq  4$ and $P_8\leq 5$ in next steps.
\medskip

{\bf Step 2}. Case (\ref{423}.1).

In Property {\bf  ($\aleph$-2)}, $P_5=3$ implies $\sigma_5=0$. From $\varepsilon_6=0$, we get  $4\geq P_7=P_6+1$. Thus $P_6=P_5=3$ and $P_8=4,5$.
Referring to the corresponding situation in the previous proposition, we have proved $\xi\geq \frac{1}{3}$.  Thus $K_X^3\geq \frac{1}{4\cdot 5}\xi\geq \frac{1}{60}$.
Since $\varepsilon_7=1,2$, we have either
$$B^{(7)}=\{3\times (1,2), 2\times (3,7), 2\times (2,5), (1,3)\}$$
with $K^3=\frac{1}{210}$ (contradicting to  $K_X^3\geq \frac{1}{60}$);  or
$$B^{(7)}=\{4\times (1,2), (3,7), 3\times (2,5), (1,3)\}$$
with $K^3=\frac{2}{105}$, the only possible packing is
$$B_{60}=\{4\times (1,2), (5,12), 2\times (2,5), (1,3)\}$$
with $K^3=\frac{1}{60}$ which is minimal. This asserts (i) and (ii). 
\medskip

{\bf Step 3}. Case (\ref{423}.2).

We still have $\xi\geq\frac{1}{3}$ and so $K_X^3\geq \frac{1}{60}$. Similarly, since
$$5\geq P_8\geq P_6+P_2.$$
So we have $P_8=5$ and
$$B^{(7)}=\{7\times (1,2), (3,7), 2\times (1,3), (2,7)\}$$
with $K^3=\frac{1}{42}$. This has only one possible packing
$$B_{105}=\{7\times (1,2), (3,7), (1,3), (3,10)\}$$
with $K^3=\frac{2}{105}$. This asserts (iii) and (iv).
\medskip

{\bf Step 4}. Case (\ref{423}.3).

We have proved that
$$B_X=B^{(5)}=\{7\times (1,2), 2\times (2,5), 2\times (1,3), (1,4)\}$$ and $K_X^3=\frac{1}{60},$
which asserts (v).
\end{proof}

\begin{thm}\label{t2} Let $X$ be a minimal projective 3-fold of general type with $p_g(X)=1$ and $P_4(X)\geq 2$. Then
\begin{itemize}
\item[(1)] $\varphi_{17,X}$ is birational;
\item[(2)] $\varphi_{16,X}$ is birational unless $X$ belongs to one of the following types:
\begin{itemize}
\item[(i)] $B_X=\{4\times (1,2), (3,7), 3\times (2,5), (1,3)\}$,
 $K^3=\frac{2}{105}$, $P_2(X)=1$ and $\chi(\OO_X)=1$;
 \item[(ii)] $B_X=\{4\times (1,2), (5,12), 2\times (2,5), (1,3)\}$,
$K^3=\frac{1}{60}$, $P_2(X)=1$ and $\chi(\OO_X)=1$;

\item[(iii)] $B_X=\{7\times (1,2), (3,7), 2\times (1,3), (2,7)\}$,
$K^3=\frac{1}{42}$, $P_2(X)=1$ and $\chi(\OO_X)=1$;

\item[(iv)] $B_X=\{7\times (1,2), (3,7), (1,3), (3,10)\}$,
$K^3=\frac{2}{105}$, $P_2(X)=1$ and $\chi(\OO_X)=1$;

\item[(v)]  $B_X=\{7\times (1,2), 2\times (2,5), 2\times (1,3), (1,4)\}$, $K^3=\frac{1}{60}$,
 $P_2(X)=1$ and $\chi(\OO_X)=1$.
\end{itemize}
\end{itemize}
\end{thm}
\begin{proof} Theorem \ref{t2} follows directly from Proposition \ref{421}, Proposition \ref{422} and Proposition \ref{422}.
\end{proof}

Theorem \ref{t1} and Theorem \ref{t2} imply Theorem \ref{main1}.

\section{\bf Threefolds of general type with $p_g=3$ (Part I)}\label{pg3d=2}

Within this section, we assume $p_g(X)=3$, $d_1=2$ and keep the same set up as in \ref{setup}.  The general fiber $C$ of the induced fibration $f:X'\lrw \Gamma$ is a curve of genus $g\geq 2$. Let us recall the following theorem.

\begin{thm}\label{Mm1}
Let $X$ be a minimal projective 3-fold of general type. Assume $p_g(X)=3$.  Then
\begin{itemize}
\item[(i)] (\cite[Theorem 1.5(1)]{MA}) $K_X^3\geq 1$.

\item[(ii)]  (\cite[Theorem 4.1]{C14}) when $K_X^3>1$ and $d_1=2$,  $\varphi_{5,X}$ is birational.
\end{itemize}
\end{thm}

In fact, by the argument in  \cite[3.2]{MA}, $K_X^3=1$ implies $g(C)=2$ and $\xi=(\pi^*(K_X)\cdot C)=1$.

{}From now on within this section, we always assume that $K_X^3=1$.  Take $|G|=|S|_S|$, which means $\beta\geq 1$. Since
\begin{equation}1=K_X^3\geq (\pi^*(K_X)|_S\cdot S|_S)\geq (\pi^*(K_X)|_S\cdot \beta C)\geq \beta,
\label{122}\end{equation}
 we see $\beta=1$.  This also implies that $|G|$ is composed of a free rational pencil on $S$ and that $h^0(G)=h^0(C)=2$.
Recall that $\sigma: S\lrw S_0$ is the contraction onto the minimal model.
By Theorem \ref{KaE}, we have
\begin{equation*} |4K_{X'}||_S\lsgeq |2(K_{X'}+S)||_S=|2K_S|\lsgeq |2\sigma^*(K_{S_0})|,
\end{equation*}
which directly implies
\begin{equation}\label{eq_KS0}
\pi^*(K_X)|_S \geq \frac{1}{2} \sigma^*(K_{S_0}).
\end{equation}

\begin{lem}\label{lem_pg}  Assume $p_g(X)=3$, $d_1=2$ and $K_X^3=1$.
Then $$p_g(S)=4,\  q(S)=0$$ and $K^2_{S_0}= 4$.
\end{lem}
\begin{proof}
We have $K_S\sim (K_{X'}+S)|_S\geq 2C$ and $h^0(S,C)=2$. Hence $p_g(S)\geq h^0(S,2C)\geq 3$.
By \eqref{eq_KS0}, we have $K^2_{S_0}\leq 4$. On the other hand, the Noether inequality (see \cite[Chapter VII.3]{BPV}) implies $p_g(S)=p_g(S_0) \leq 4$. Finally, by Debarre  \cite{D}, we obtain $q(S)=q(S_0)=0$. By the Noether inequality and $K_{S_0}^2\le 4$, it is sufficient to prove that $p_g(S)=4$.

Suppose that  $p_g(S)\neq 4$. Then we have $p_g(S)=3$. Note that we have $K_S\geq \sigma^*(K_{S_0})\geq 2C$ and $h^0(S,2C)=3$. So $|K_{S_0}|$ is composed of a pencil of curves.  By Lemma \ref{ll3}, we have $$(K_{S_0}\cdot \sigma_*(C))=(\sigma^*(K_{S_0})\cdot C)\geq 2. $$  
Since $g(C)=2$, we have $(\sigma^*(K_{S_0})\cdot C)\le (K_S\cdot C)=2$. We deduce that $(K_{S_0}\cdot \sigma_*(C))=2$. Since $K_{S_0}\ge 2\sigma_*(C)$, we have 
$$
K_{S_0}^2\ge 2(K_{S_0}\cdot\sigma_*(C))\ge 4.
$$
By Hodge index theorem, we have $\sigma_*(C)^2\le 1$. Note that $((K_{S_0}+\sigma_*(C))\cdot\sigma_*(C))$ is a positive even integer by adjunction formula. Thus we have $\sigma_*(C)^2=0$ and $((K_{S_0}+\sigma_*(C))\cdot\sigma_*(C))=2$ . In particular, the linear system $|\sigma_*(C)|$ is base point free and is composed of a pencil of curves of genus $2$, i.e., $|\sigma_*(C)|$ is composed of a free pencil of curves of genus $2$. 
Thus $\Mov|K_{S_0}|$ is composed of a free pencil of curves of genus $2$. By Xiao (\cite[Chapter 5, Corollaire 1]{X85}), one has $K^2_{S_0} \geq 4p_g(S_0)-6\geq 6$, a contradiction.

\end{proof}


\begin{prop}\label{pg444} Under the same condition as that of Lemma \ref{lem_pg}, Assume $p_g(S)=4$.  Then $|K_{S_0}|$ induces a double cover $\tau \colon S_0 \longrightarrow \mathbb{F}_2$ ($\mathbb{F}_2$ denotes the Hirzebruch ruled surface with a $(-2)$-section) and $\varphi_{5,X}$ is non-birational.
\end{prop}
\begin{proof}  Clearly we have $K_{S_0}^2=4$ by the Noether inequality.
By our assumption, $|C|$ is a free pencil on $S$ and $\sigma^*(K_{S_0})\geq 2C$.  If $|K_S|$ is composed of a pencil of curves, then $\Mov|K_S|=|3C|$. Hence
$$\pi^*(K_X)|_S\geq \frac{1}{2}\sigma^*(K_{S_0})\geq \frac{3}{2}C,$$
which means $\beta\geq \frac{3}{2}$, a contradiction. So $|K_S|$ is not composed of a pencil of curves.  In fact, such surfaces have been classified by Horikawa (see \cite[Theorem 1.6(iii),(iv)]{H1}).  Namely, $S$ belongs to one of the following types:
\begin{enumerate}
\item the canonical map $\Phi_{|K_{S_0}|}$ gives a double cover of $S_0$ onto $\bP^1 \times \bP^1$ whose branch locus is linearly equivalent to $6l_1+6l_2$, where $l_1$ and $l_2$ are two natural line classes on $\bP^1 \times \bP^1$ with $l_1^2=l_2^2=0$;
\item $\Phi_{|K_{S_0}|}$ induces a double cover  $\tau \colon S_0 \longrightarrow \mathbb{F}_2$ whose branch locus is linearly equivalent to $6\Delta_0+12\gamma$, where $\Delta_0$ is the unique section with $\Delta_0^2=-2$ and $\gamma$ is a fibre of the ruling of $\mathbb{F}_2$ with $\gamma^2=0$.
\end{enumerate}
\medskip

\noindent{\bf Case (1) is impossible}.  By the ramification formula, one has
$K_{S_0} \equiv C_1 + C_2$, where $C_i$ is the pullback of $l_i$ for $i=1,2$. On the other hand, we have a genus 2 curve class $\widehat{C}=\sigma_*C$.  With the similar reason to that in the proof of Lemma \ref{lem_pg}, $|\widehat{C}|$ is a free pencil on $S_0$.  Noting that $K_{S_0}\geq 2\widehat{C}$, we have
$$2=(K_{S_0}\cdot \widehat{C})=(C_1\cdot \widehat{C})+(C_2\cdot \widehat{C}).$$
Here we have three free pencils of curves of genus 2. If $C_i\not\equiv \widehat{C}$ for some $i$, then $(C_i\cdot\widehat{C})\geq 2$ as $\widehat{C}$ is moving on $C_i$.  So the only possibility is that $C_1\equiv \widehat{C}$ while $(C_2\cdot \widehat{C})=2$.
But then one has
$$2=(K_{S_0}\cdot C_2)\geq 2(\widehat{C}\cdot C_2)\geq 4,$$
a contradiction.
\medskip

\noindent{\bf Case (2) implies the non-birationality of $\varphi_{5,X}$}.  By \eqref{122} we have $(\pi^*(K_X)|_S)^2=1$. On the other hand, we have
$$2=2(\pi^*(K_X)|_S)^2\geq (\sigma^*(K_{S_0})\cdot \pi^*(K_X)|_S)\geq \sqrt{K_{S_0}^2\cdot (\pi^*(K_X)|_S)^2}=2,$$
which means, by the Hodge Index Theorem, that
\begin{equation}\label{eqqq}\pi^*(K_X)|_S\equiv \frac{1}{2}\sigma^*(K_{S_0}).
\end{equation}
According to Horikawa, the double cover $\tau\colon S_0 \longrightarrow \mathbb{F}_2$ is branched over a smooth divisor $B \in |6 \Delta_0 + 12\gamma|$. By construction $(\Delta_0 \cdot B)=0$ and $\tau^* \Delta_0=A_1 + A_2$ with $A_i^2=-2$ for $i=1,2$ and $(A_1\cdot A_2)=0$. Denote by $C_0=\tau^*\gamma$. Then, by the ramification formula, we have $K_{S_0}\sim 2C_0+A_1+A_2$.
Let us pullback everything to $S$ and take $\tilde{C_0}=\sigma^*(C_0)$, $\tilde{A_i}=\sigma^*(A_i)$ for $i=1,2$.  Then $\sigma^*(K_{S_0})\sim 2\tilde{C_0}+\tilde{A_1}+\tilde{A_2}$.
For the similar reason, we see $C\equiv \tilde{C_0}$ since $(\sigma^*(K_{S_0})\cdot C)=2$.  Thus $C$ and $\tilde{C_0}$ are in the same curve class.  Thus we have
\begin{equation}\label{eqpq}
\pi^*(K_X)|_S\equiv C+\frac{1}{2}(\tilde{A_1}+\tilde{A_2}).\end{equation}
Denote by $\hat{A_i}$ ($i=1,2$) the strict transform of $A_i$ on $S$.  Then $(\sigma^*(K_{S_0})\cdot \hat{A_i})=0$ for $i=1,2$ since
$(\sigma^*(K_{S_0})\cdot \tilde{A_i})=0$.

Let us denote by $\iota$ the restriction map $f|_S:S\lrw f(S)$. The general fiber of $\iota$
is in the same class of $C$.  Since $\pi^*(K_X)\geq S$, we may write $\pi^*(K_X)=\hat{S}+E_1'$ where $\hat{S}$ is certain special member of $|M|$ and $E_1'$ is an effective $\bQ$-divisor. Denote by $C'=\hat{S}|_S$. Clearly $C'\sim C$. Then
$$\pi^*(K_X)|_S=C'+J_v+J_h$$
where $J_v$ and $J_h$ are effective $\bQ$-divisor, $J_v$ is vertical with respect to $\iota$ while $J_h$ is horizontal with respect to $\iota$. Since $\pi^*(K_X)|_S\leq K_S$ and $(K_S\cdot C)=2$, $J_h$ has at most two irreducible components. Suppose $\hat{A_i}$ is not any component of $\text{Supp}(J_h)$. Then
$$0=(\hat{A_i}\cdot \pi^*(K_X)|_S)\geq (\hat{A_i}\cdot C')=(\hat{A_i}\cdot \tilde{C}_0)=(A_i\cdot C_0)>0, $$
a contradiction.  Hence it asserts that $J_h=a\hat{A_1} +b\hat{A_2}$ with  $a, b>0$ and
$a+b=1$.  Now since $\sigma_*(\pi^*(K_X)|_S)\equiv \frac{1}{2}K_{S_0}$ and the $A_i$ is a horizontal $(-2)$-curve, one gets $a\geq \frac{1}{2}$ and $b\geq \frac{1}{2}$, whence $a=b=\frac{1}{2}$.  In a word, we see that
\begin{equation}
5\pi^*(K_X)|_S=5C'+\frac{5}{2}(\hat{A_1}+\hat{A_2})+5J_v.\label{5KK}
\end{equation}
Since
$$M_5|_S\leq \rounddown{5C'+\frac{5}{2}(\hat{A_1}+\hat{A_2})+5J_v}
=5C'+2\hat{A_1}+2\hat{A_2}+\rounddown{5J_v}$$
and $\rounddown{5J_v}$ is vertical with respect to $f|_S$,
we see that $(M_5|_S\cdot C)\leq 4$.  On the other hand, by our assumption, $(\hat{A_1}+\hat{A_2})|_C\sim K_C$.  By the Kawamata-Viehweg vanishing theorem once more,  we get the following two relations:
\begin{eqnarray}
|5K_{X'}||_S&\lsgeq&|K_{X'}+\roundup{3\pi^*(K_X)}+S||_S\notag\\
&\lsgeq&|K_S+\roundup{3\pi^*(K_X)|_S}|\notag\\
&\lsgeq&|K_S+\roundup{3\pi^*(K_X)|_S-\frac{1}{2}(\hat{A_1}+\hat{A_2})-J_v}|\label{mee1}
\end{eqnarray}
and
\begin{eqnarray}
&&|K_S+\roundup{3\pi^*(K_X)|_S-\frac{1}{2}(\hat{A_1}+\hat{A_2})-J_v}||_C\notag\\
&=&|K_C+(\hat{A_1}+\hat{A_2})|_C|=|2K_C|.\label{mee2}
\end{eqnarray}
By \eqref{mee1} and \eqref{mee2}, we have $|M_5||_C\lsgeq |2K_C|$. Note that $(M_5|_S\cdot C)\le 4$. We deduce that
$|M_5||_C=|2K_C|$. Since $C$ is a smooth curve of genus $2$, $|2K_C|$ is not birational. Note that the curve class parameterized by $C$ covers $X'$. Hence $\varphi_{5,X}$ is not birational.
\end{proof}

\begin{thm}\label{thm_d2}
Let $X$ be a minimal projective 3-fold of general type with  $p_g(X)=3$ and $d_1 = 2$. Then $\varphi_{5,X}$ is non-birational if and only if $K^3_X=1$.
 \end{thm}
\begin{proof} This theorem follows directly from Theorem \ref{Mm1}, Lemma \ref{lem_pg} and Proposition \ref{pg444}.
\end{proof}

Theorem \ref{thm_d2} is sharp and here is an example due to Iano-Fletcher \cite{Flet}:

\begin{exmp}\label{ex_pg3d2}
The general hypersurface $X=X_{12} \subset \bP(1, 1, 1, 2, 6)$ of degree $12$ has the invariants $p_g = 3$ and $K^3_X = 1$, but $\varphi_{5,X}$ is non-birational.
Notice that in this example $X_{12}$ is a double cover of $\mathbb{P}(1,1,1,2)$ ramified over a sextic. The surface $S$ maps $2:1$ onto $\bP(1,1,2)$, which exactly fits into the situation described in the proof of Proposition \ref{pg444}\end{exmp}


\section{\bf Threefolds of general type with $p_g=3$  (Part II)}

This section is devoted to studying the case $p_g(X)=3$ and $d_1=1$.   Keep the same notation as in \ref{setup}.
We have an induced fibration $f\colon  X' \longrightarrow \Gamma$ of which the general fiver $F$ is a nonsingular projective surface of general type.  Let $\sigma\colon F \longrightarrow F_0$ be the contraction onto the minimal model.

By \cite[Theorem 3.3]{IJM}, it is sufficient to assume $b=g(\Gamma)=0$, i.e. $\Gamma \cong \bP^1$. Note that $p_g(X) > 0$ implies $p_g(F)>0$.  Thus $F_0$ must be among the following types by the surface theory:
\begin{enumerate}
\item[(1)] $(K_{F_0}^2 ,p_g(F_0))=(1,2)$;
\item[(2)] $(K_{F_0}^2 ,p_g(F_0)) = (2,3)$;
\item[(3)]  other surfaces with $p_g(F_0) > 0$.
\end{enumerate}
By \cite[Theorem 4.3 and Claims 4.2.1, 4.2.2]{C14} it suffices to consider Case (1). It is well known that, for a $(1,2)$-surface, $|K_{F_0}|$ has one base point and that, after blowing up this point, $F$ admits a canonical fibration with a unique section which we denote by $H$. Denote by $C$ a general member in  $|G|=\Mov|\sigma^*(K_{F_0})|$. Set $m_0=1$.

\subsection{Several sufficient conditions for the birationality of $\varphi_{5,X}$}

\begin{lem}\label{M1} Let $X$ be a minimal projective 3-fold of general type with $p_g(X)=3$, $d_1=1$, $\Gamma\cong \bP^1$. Assume that $F$ is a $(1,2)$-surface. Then $\beta(1,|G|)\geq \frac{2}{3}$, $\xi(1, |G|)=1$ and $(\pi^*(K_X)|_F)^2\geq \frac{2}{3}$.
\end{lem}
\begin{proof} By our definition in \ref{setup}, one has $\zeta(1)=2$ and
\begin{equation}\label{eq_K2F}
\pi^*(K_X) \sim_{\bQ}2F+E'_1
\end{equation}
where $E'_1$ is an effective $\bQ$-divisor.
By \eqref{cri} (or, see \cite[Corollary 2.5]{MZ})
\begin{equation}\label{eq_p}
\pi^*(K_X)|_F\equiv\frac{2}{3}\sigma^*(K_{F_0})+Q'
\end{equation}
where $Q'$ is an effective $\bQ$-divisor on $F$. In particular, we have
$$\beta(1,|G|)\geq \frac{2}{3}. $$ This also implies that
\begin{equation}(\pi^*(K_X)|_F)^2\geq \frac{2}{3}(\pi^*(K_X)|_F\cdot C)=\frac{2}{3}\xi.\label{2/3}\end{equation}
Finally we know $\xi\geq1$ by \cite[Claim 4.2.3]{C14}. As it is clear that $\xi\leq (\sigma^*(K_{F_0})\cdot C)=1$, one has $\xi=\xi(1,|G|)=1$.
\end{proof}

\begin{lem}\label{lem_beta} Under the same condition as that of Lemma \ref{M1}, if $\beta(1, |G|) > \frac{2}{3}$, then $\varphi_{5,X}$ is birational.
\end{lem}

\begin{proof} Since
\[ \alpha_5\geq (5-1-\frac{1}{\mu}-\frac{1}{\beta})\cdot \xi>2,
\]
$\varphi_{5,X}$ is birational by Lemma \ref{s1}, Lemma \ref{s2} and Theorem \ref{key-birat}.
\end{proof}

By Equality \eqref{eq_p}, we may write
\begin{equation}\pi^*(K_X)|_F\equiv \frac{2}{3}C+E_F\label{eq_c}\end{equation}
where $E_F$ is an effective $\bQ$-divisor on $F$.

\begin{lem} \label{lem_KXquadro} Under the same condition as that of Lemma \ref{M1}, if
\[ (\pi^*K_X|_F)^2>\frac{2}{3},
\]
 then $\varphi_{5,X}$ is birational.
\end{lem}

\begin{proof}
Consider  the Zariski decomposition of the following $\bQ$-divisor:
\[ 2\pi^*(K_X)|_F +  \frac{3}{2}E_F=(2\pi^*(K_X)|_F +N^+) +N^-,
\]
where
\begin{enumerate}
\item[(z1)] both $N^+$ and $N^-$ are effective $\bQ$-divisors and $N^++N^-=\frac{3}{2}E_F$;
\item[(z2)] the $\bQ$-divisor $\pi^*(K_X)|_F +N^+ $ is nef;
\item[(z3)] $((\pi^*(K_X)|_F +N^+) \cdot N^-)= 0$.
\end{enumerate}
\medskip

\noindent{\bf Step 1}.  $(\pi^*K_X|_F)^2>\frac{2}{3}$ implies $(N^+ \cdot C) > 0$.

Since $C$ is nef, we see $(N^+ \cdot C) \geq 0$. Assume the contrary that $(N^+ \cdot C) =0$.  Then $(N^+)^2\leq 0$ as $C$ is a fiber of the canonical fibration of $F$.
Since
$$\frac{2}{3}<(\pi^*(K_X)|_F)^2=\frac{2}{3}(\pi^*(K_X)\cdot C)+(\pi^*(K_X)|_F\cdot E_F)$$
implies $(\pi^*(K_X)|_F\cdot E_F)>0$, we clearly have $(\pi^*(K_X)|_F\cdot N^+)>0$ by the definition of Zariski decomposition.
Now
\begin{eqnarray*}
(N^+)^2&=&\big(N^+\cdot (\frac{3}{2}\pi^*(K_X)|_F-C-N^-)\big)\\
&=&\frac{3}{2}(N^+\cdot \pi^*(K_X)|_F)+(\pi^*(K_X)|_F\cdot N^-)>0,
\end{eqnarray*}
a contradiction.
\medskip

\noindent{\bf Step 2}. $(N^+\cdot C)>0$ implies the birationality of $\varphi_{5,X}$.

By the Kawamata-Viehweg vanishing theorem, we have
\begin{eqnarray*}
|5K_{X'}||_F&\lsgeq&|K_{X'}+\roundup{4\pi^*(K_X)-\frac{1}{2}E_1'}||_F\\
&\lsgeq& |K_F+\roundup{(4\pi^*(K_X)-\frac{1}{2}E_1')|_F}.
\end{eqnarray*}
Noting that
\begin{eqnarray}
(4\pi^*(K_X)-\frac{1}{2}E_1')|_F&\equiv& \frac{7}{2}\pi^*(K_X)|_F\equiv 2\pi^*(K_X)|_F+C+\frac{3}{2}E_F\notag\\
&\equiv& (2\pi^*(K_X)|_F+N^+)+C+N^-,\label{d201}
\end{eqnarray}
and that $2\pi^*(K_X)|_F+N^+$ is nef and big, the vanishing theorem gives
\begin{equation}
|K_F+\roundup{(4\pi^*(K_X)-\frac{1}{2}E_1')|_F-N^-}||_C=|K_C+D^+|,\label{d202}
\end{equation}
where $\deg(D^+)\geq 2\xi+(N^+\cdot C)>2$.  By Lemma \ref{s1}, Lemma \ref{s2}, \eqref{d201} and \eqref{d202}, $\varphi_{5,X}$ is birational.
\end{proof}

\begin{lem}\label{indeX} Under the same condition as that of Lemma \ref{M1}, if the Cartier index $r_X$ is not divisible by $3$, $\varphi_{5,X}$ is birational.
\end{lem}
\begin{proof}  By \cite[Lemma 2.1]{Delta18},  we see that $r_X(\pi^*(K_X)|_F)^2$ is an integer. When $r_X$ is not divisible by $3$, one has
$$(\pi^*(K_X)|_F)^2\geq \frac{\roundup{\frac{2}{3}r_X}}{r_X}>\frac{2}{3}.$$ Thus $\varphi_{5,X}$ is birational according to \eqref{2/3} and Lemma \ref{lem_KXquadro}.
\end{proof}

\begin{lem}\label{LPm1}  Let $m_1\geq 2$ be any integer. Under the same condition as that of Lemma \ref{M1},
$\varphi_{5,X}$ is birational provided that one of the following holds:
\begin{enumerate}
\item[(i)] $u_{m_1,0}=h^0(F, m_1K_F)$;
\item[(ii)] $h^0(M_{m_1}-jF)\geq 2m_1-j+2>1$ and $u_{m_1, -j}\leq 1$ for some integer $j\geq 0$.
\end{enumerate}
\end{lem}
\begin{proof} (i). Since $\theta_{m_1,0}$ is surjective and $|m_1\sigma^*(K_{F_0})|$ is base point free, we have
$$m_1\pi^*(K_X)|_F\geq M_{m_1}|_F\geq m_1\sigma^*(K_{F_0})$$
which means that $\beta=1$. By Lemma \ref{lem_beta}, $\varphi_{5,X}$ is birational.

(ii).  By assumption, $|M_{m_1}-jF|$ and $|F|$ are composed of the same pencil. Hence we have $M_{m_1}\geq (2m_1+1)F$, which means  $\mu\geq \frac{2m_1+1}{m_1}$. By  \eqref{cri}, we get
$$\pi^*(K_X)|_F\geq \frac{2m_1+1}{3m_1+1}\sigma^*(K_{S_0})$$
which means $\beta>\frac{2}{3}$.  By Lemma \ref{lem_beta}, $\varphi_{5,X}$ is birational.
\end{proof}

\subsection{The solvability of explicit classification assuming the non-birationality of $\varphi_{5,X}$}\  

Now we will apply the results in Subsection \ref{3.2} to do further discussion.

Recall from Definition \ref{twomaps}, for any integers $j> 0$ and $m_1>1$, one has 
\begin{equation}
P_{m_1}(X)=h^0(M_{m_1}-(j+1)F)+u_{m_1,0}+u_{m_1,-1}+\cdots+u_{m_1,-j}. \label{Ne1}
\end{equation}
By Proposition \ref{-jl} and Lemma \ref{LPm1}, we may assume that 
\begin{equation} u_{m_1,0}\leq P_{m_1}(F)-1=\frac{1}{2}m_1(m_1-1)+2.\label{Ne2}
\end{equation}

\begin{lem}\label{Bn1}  Let $m_1\geq 2$ be an integer. Keep the same condition as that of Lemma \ref{M1}.  Assume that $\varphi_{5,X}$ is non-birational. Then 
$$P_{m_1}\leq \frac{1}{2}jm_1(m_1-1)+2j$$ 
holds for any integer $j>2m_1-1$. In particular, one has
$$P_{m_1}\leq m_1^3-m_1^2+4m_1.$$ 
\end{lem}
\begin{proof} Assume that we have $P_{m_1}> \frac{1}{2}jm_1(m_1-1)+2j$. 
By Equation \eqref{Ne1} and Inequality \eqref{Ne2}, we have 
$h^0(M_{m_1}-(j+1)F)>0$ which means $M_{m_1}\geq (j+1)F$. 
Inequality \eqref{cri}, we have
$\pi^*(K_X)|_F\geq \frac{j+1}{m_1+j+1}\sigma^*(K_{F_0})$ which implies 
$\beta(1, |G|)>\frac{2}{3}$. By virtue of Lemma \ref{lem_beta}, $\varphi_{5,X}$ is birational, a contradiction.  Hence the lemma is proved. 

In particular, take $j=2m_1$, we get $P_{m_1}\leq m_1^3-m_1^2+4m_1.$
\end{proof}

\begin{rem} The key role of Lemma \ref{Bn1} is that, if $\varphi_{5,X}$ is non-birational, then $P_{m_1}$ is upper bounded for any $m_1>1$.  For instance, we have $P_2\leq 12$. In fact, Subsection \ref{3.2}, Lemma \ref{lem_beta} 
and Lemma \ref{lem_KXquadro} allow us to get effective upper bounds for $P_{m_1}$ ($2\leq m_1\leq
6$), which are essential in our explicit classification. 
\end{rem} 

Just to illustrate the main idea of our explicit study, we present here the following result for the case $m_1=2$:

\begin{prop}\label{VP2}  Keep the same condition as that of Lemma \ref{M1}. Assume that $\varphi_{5,X}$ is non-birational. Then $P_2(X)\leq 8$. 
\end{prop}
\begin{proof}  Suppose, to the contrary, that $P_2(X)\geq 9$. Set $m_1=2$.
By virtue of Lemma \ref{LPm1}, we may assume $u_{2,0}\leq h^0(2K_F)-1=3$.
\medskip

{\bf Case 1}. $u_{2,-1}=3$.

There is a moving divisor $S_{2,-1}$ on $X'$ such that
$$M_2\geq F+S_{2,-1}$$
and $h^0(X', S_{2,-1}|_F)\geq 3$. Modulo further birational modification, we may and do assume that $|S_{2,-1}|$ is base point free.  Denote by $C_{2,-1}$ the generic irreducible element of $|S_{2,-1}|_F|$.  Then $|C_{2,-1}|$ is moving as $q(F)=0$.

If $|S_{2,-1}|_F|$ and $|C|$ are composed of the same pencil,  then
$$M_{2}|_F\geq S_{2,-1}|_F\geq 2C$$
which means that $\beta\geq 1$.  By Lemma \ref{lem_beta}, $\varphi_{5,X}$ is birational.

If $|S_{2,-1}|_F|$ and $|C|$ are not composed of the same pencil (which implies that $(C_{2,-1}\cdot C)\geq 2$),  Proposition \ref{X1}(1) implies that $\varphi_{5,X}$ is birational.

\medskip

{\bf Case 2}.  $u_{2,-1}\leq 2$ and $u_{2,-2}=2$

We have
$$M_2\geq 2F+S_{2,-2}$$
where $S_{2,-2}$ is a moving divisor on $X'$ with $h^0(F, S_{2,-2}|_F)\geq 2$.  Similarly we may and do assume that $|S_{2,-2}|$ is base point free modulo further birational modifications.  When $|S_{2,-2}|_F|$ and $|C|$ are not composed of the same pencil, Proposition \ref{X1}(1) implies the birationality of $\varphi_{5,X}$.  When  $|S_{2,-2}|_F|$ and $|C|$ are composed of the same pencil,  Theorem \ref{k1} ($n_1=j_1=2$, $l_1=1$) implies $\beta(m_0, |C|)\geq \frac{3}{4}>\frac{2}{3}$. By Lemma  \ref{lem_beta}, $\varphi_{5,X}$ is birational.
\medskip

{\bf Case 3}. $u_{2,-1}\leq 2$, $u_{2,-2}\leq 1$ and $P_2(X)\geq 9$.

 Clearly, $h^0(M_2-2F)\geq 4$.  By Lemma \ref{LPm1} ($m_1=j=2$), $\varphi_{5,X}$ is birational.
\end{proof}

By the similar method, but slightly more complicated arguments, one should have no technical difficulties to obtain the following proposition, for which we omit the proof in details:
\medskip

\noindent{\bf Proposition X}. {\em Keep the same condition as that of Lemma \ref{M1}. Assume that $\varphi_{5,X}$ is non-birational. Then $P_3(X)\leq 15$, $P_4(X)\leq 26$, $P_5(X)\leq 41$ and $P_6(X)\leq 63$. 
Moreover,
when $P_3(X)=15$ or $P_4(X)=26$ or $P_5(X)=41$, $\varphi_{5,X}$ is non-birational if and only if  $$(\pi^*(K_X)|_F)^2=\frac{2}{3}.$$
}

\medskip

We would like to explain the outline for classifying the weighted basket ${\mathbb B}(X)$.  
Keep the same condition as that of Lemma \ref{M1} and assume that $\varphi_{5,X}$ is non-birational.
Then the following holds: 
\begin{itemize}
\item[(c1)]  $\chi(\OO_X)=-1$ or $-2$ since $q(X)=0$, $h^2(\OO_X)=0,\ 1$ and $p_g(X)=3$;

\item[(c2)]  $6\leq P_2(X)\leq 8$, $P_3(X)\leq 14$, $P_4(X)\leq 25$, $P_5(X)\leq 40$, $P_6(X)\leq 63$;

\item[(c3)] $K_X^3\geq \frac{4}{3}$ by \cite[3.7]{MA};



\item[(c4)]  $r_X$ is $3$-divisible, which applies to the basket $B_X$ rather than $B^{(5)}$.
\end{itemize}

The above situation naturally fits into the hypothesis of \cite[(3.8)]{EXP1} from which we can
list all the possibilities for $B^{(5)}(X)$.  To be precise,
{\footnotesize
$$B^{(5)}=\{ n^5_{1,2} \times (1,2),n^5_{2,5} \times (2,5), n^5_{1,3}
\times (1,3), n^5_{1,4}  \times (1,4), n^5_{1,5} \times
(1,5),\cdots\}$$} with
$$B^{(5)}\left\{
 \begin{array}{l}
  n^5_{1,2}=3 \chi(\OO_X) +6P_2- 3 P_3 + P_4 - 2 P_5 + P_6+ \sigma_5 ,\\
 n^5_{2,5}= 2 \chi(\OO_X)-P_3 +   2 P_5  - P_6- \sigma_5 \\
 n^5_{1,3}= 2 \chi(\OO_X)+2P_2+ 3 P_3- 3 P_4 -P_5 + P_6   + \sigma_5 , \\
 n^5_{1,4}= \chi(\OO_X) -3P_2+  P_3 +2 P_4 -P_5-\sigma_5 \\
  n^5_{1,r}=n^0_{1,r}, r \geq 5
\end{array} \right. $$
where $\sigma_5=\sum_{r\geq 5}n_{1,r}^0\geq 0$ and
$$\sigma_5\leq 2\chi(\OO_X)-P_3+2P_5-P_6.$$
Note also that, by our definition, each of the above coefficients satisfies $n_{*,*}^0\geq 0$.
With all these constraints, a computer program outputs a raw list of about 500 possibilities for $\{B_X^{(5)}, P_2(X),\chi(\OO_X)\}$. Taking into account those possible packings, we have the following conclusion.

\begin{cor}\label{list_pg3}  Let $X$ be a minimal projective 3-fold of general type with $p_g(X)=3$, $d_1=1$, $\Gamma\cong \bP^1$. Assume that $F$ is a $(1,2)$-surface and that $\varphi_{5,X}$ is non-birational. Then ${\mathbb B}(X)$ corresponds to one element of certain concrete finite set ${\mathbb S}_3$. 
\end{cor}

Being aware of the length of this paper, we do not list the set ${\mathbb S}_3$, which can be found, however, at
\begin{center}
   \verb|http://www.dima.unige.it/~penegini/publ.html|
\end{center}
\medskip
 
{}Finally it is clear that Theorem \ref{thm_pg3} follows from Theorem \ref{thm_d2},  
\cite[Theorem 4.3 and Claims 4.2.1, 4.2.2]{C14} and Corollary \ref{list_pg3}. 

\medskip

{\bf Acknowledgment}.  This work was partially supported by Key Laboratory of Mathematics for Nonlinear Sciences, Fudan University.  The authors would like to thank the referee for valuable comments which greatly improves the expression of this paper.


\end{document}